\documentclass[a4paper,10pt]{article}
\usepackage{general}
\usepackage{amsmath,amsfonts}
\usepackage{fullpage} 
\usepackage{authblk}
\usepackage{booktabs}
\usepackage{cite}
\usepackage{comment}
\usepackage{enumerate}
\usepackage{hyperref}
\usepackage[bold]{hhtensor}
\usepackage[latin1]{inputenc}
\usepackage{paralist}

\numberwithin{equation}{section}

\usepackage{pgfplots,pgfplotstable}
\usepackage{subcaption}

\usepackage{caption}
\usepackage{tabularx}
\usepackage{xcolor}

\usetikzlibrary{external}

\newcommand{\email}[1]{\href{mailto:#1}{#1}}

\graphicspath{{./fig/pdf/}}

\newcommand{\UT}[1][k]{\underline{\mathsf{U}}_T^{#1}}
\newcommand{\Uh}[1][]{\underline{\mathsf{U}}_{h#1}^k}
\newcommand{\UhD}[1][0]{\underline{\mathsf{U}}_{h,#1}^k}
\newcommand{\Uhzero}[1][]{\underline{\mathsf{U}}_{h#1}^0}

\newcommand{\su}[1][T]{\underline{\mathsf{u}}_{#1}}
\newcommand{\sv}[1][T]{\underline{\mathsf{v}}_{#1}}
\newcommand{\sw}[1][T]{\underline{\mathsf{w}}_{#1}}

\newcommand{\sg}[1][T]{\underline{\mathsf{g}}_{#1}}

\newcommand{\unu}[1][]{\mathsf{u}_{#1}}
\newcommand{\unv}[1][]{\mathsf{v}_{#1}}
\newcommand{\unw}[1][]{\mathsf{w}_{#1}}

\newcommand{\bphi}{\vec{\phi}}

\newcommand{\pT}[1][k+1]{p_T^{#1}}
\newcommand{\PT}[1][k+1]{P_T^{#1}}
\newcommand{\ph}[1][k+1]{p_h^{#1}}

\newcommand{\GT}[1][k]{\vec{G}^{#1}_{T}}
\newcommand{\Gh}[1][k]{\vec{G}^{#1}_{h}}

\newcommand{\lproj}[2][h]{\pi_{#1}^{#2}}

\newcommand{\IT}[1][k]{\underline{\mathsf{I}}_T^{#1}}
\newcommand{\Ih}[1][k]{\underline{\mathsf{I}}_h^{#1}}

\newcommand{\jump}[2][F]{[#2]_{#1}}

\def\bfa{{\mathbf{a}}}
\def\bchi{\boldsymbol{\chi}}
\def\bL{\boldsymbol{\Lambda}}
\def\gra{\overline{a}}
\def\upa{\beta_\bfa}
\def\coera{\lambda_\bfa}
\def\div{\mathop{\rm div}}
\def\tr{\gamma}
\def\dx{{\,d\vec{x}}}
\def\dsx{{\,ds(\vec{x})}}

\newcommand{\asch}{A}

\usepackage[normalem]{ulem}
\title{A Hybrid High-Order method for Leray--Lions elliptic equations on general meshes}
\author[1]{Daniele A. Di Pietro\footnote{\email{daniele.di-pietro@umontpellier.fr}}}
\affil[1]{University of Montpellier, Institut Montp\'{e}llierain Alexander Grothendieck, 34095 Montpellier, France}
\author[2]{J\'{e}r\^{o}me Droniou\footnote{\email{jerome.droniou@monash.edu}}}
\affil[2]{School of Mathematical Sciences, Monash University, Clayton, Victoria 3800, Australia}

\begin{document}

\maketitle


\begin{abstract}
In this work, we develop and analyze a Hybrid High-Order (HHO) method for steady non-linear Leray--Lions problems.
The proposed method has several assets, including the support for arbitrary approximation orders and general polytopal meshes.
This is achieved by combining two key ingredients devised at the local level: a gradient reconstruction and a high-order stabilization term that generalizes the one originally introduced in the linear case.
The convergence analysis is carried out using a compactness technique.
Extending this technique to HHO methods has prompted us to develop a set of discrete functional analysis tools whose interest goes beyond the specific problem and method addressed in this work: (direct and) reverse Lebesgue and Sobolev embeddings for local polynomial spaces, $L^{p}$-stability and $W^{s,p}$-approximation properties for $L^{2}$-projectors on such spaces, and Sobolev embeddings for hybrid polynomial spaces.
Numerical tests are presented to validate the theoretical results for the original method and variants thereof.

\medskip

\end{abstract}

\section{Introduction}

We are interested here in the numerical approximation of the steady Leray--Lions equation
\begin{subequations}
  \label{eq:strong}
  \begin{alignat}{2}
    \label{eq:strong:PDE}
    -\div(\bfa(\cdot,u,\GRAD u)) &= f &\qquad& \text{in $\Omega$,} 
    \\
    \label{eq:strong:BC}
    u &= 0 &\qquad& \text{on $\partial\Omega$,}
  \end{alignat}
\end{subequations}
where $\Omega\subset\Real^d$, $d\ge 1$, is a polytopal bounded connected domain of boundary $\partial\Omega$, while $\bfa:\Omega\times\Real\times\Real^d\to\Real^d$ is a (possibly nonlinear) function of its arguments, for which detailed assumptions are discussed in the following section.
The homogeneous Dirichlet boundary condition~\eqref{eq:strong:BC} is considered only for the sake of simplicity (the modifications required to handle more general boundary conditions are briefly addressed in the manuscript).
This equation, which contains the $p$-Laplace equation,
appears in the modelling of glacier motion \cite{Glowinski.Rappaz:03}, 
of incompressible turbulent flows in porous media \cite{Diaz.Thelin:94}
and in airfoil design \cite{Glowinski:84}.
Our goal is to design and analyze a discretization method for problem~\eqref{eq:strong} inspired by the Hybrid High-Order (HHO) method introduced in~\cite{Di-Pietro.Ern.ea:14} in the context of a linear diffusion model problem
(see also \cite{DiPietro.Droniou.Ern:15} for degenerate advection--diffusion--reaction models).
The proposed method offers several assets:%
\begin{inparaenum}[(i)]
\item the construction is dimension-independent;
\item fairly general meshes including polytopal elements and nonmatching interfaces are supported;
\item arbitrary polynomials orders can be considered (including the case $k=0$);
\item it is efficiently parallelisable (the local stencil only connects a mesh element with its faces), and it has reduced computational cost (when solving by a first-order algorithm, the element-based unknowns can be eliminated by static condensation).
\end{inparaenum}

\smallskip

Numerical methods allowing for arbitrary-order discretizations and general meshes have received increasing attention over the last few years.
Supporting general polytopal meshes is required, e.g., in the modelling of underground flows, where degenerate elements and nonconforming interfaces account for complex geometric features resulting from compaction, erosion, and the onset of fractures or faults.
Another relevant application of polyhedral meshes is adaptive mesh coarsening~\cite{Bassi.Botti.ea:12,Antonietti.Giani.ea:13}.
The literature on arbitrary-order polytopal methods for linear diffusion problems is vast.
In this context, methods that have similarities (and differences) with the HHO method include, e.g., the Hybridizable Discontinuous Galerkin method of~\cite{Castillo.Cockburn.ea:00,Cockburn.Gopalakrishnan.ea:09} (cf. also~\cite{Cockburn.Di-Pietro.ea:15} for a precise study of its relation with the HHO method), the Virtual Element Method of~\cite{Beirao-da-Veiga.Brezzi.ea:13,Beirao-da-Veiga.Brezzi.ea:13*1,Brezzi.Falk.ea:14}, the High-Order Mimetic method of~\cite{Lipnikov.Manzini:14}, the Weak Galerkin method of~\cite{Wang.Ye:13,Wang.Ye:14}, and the Multiscale Hybrid-Mixed method of~\cite{Araya.Harder.ea:13}.

\smallskip

The finite element approximation of nonlinear diffusion problems of Leray--Lions type on standard meshes has been studied in several papers; cf., e.g,~\cite{Barrett.Liu:94,Liu.Yan:01,Glowinski.Rappaz:03}.
The literature on polytopal meshes is, however, much more scarce, and is mainly restricted to the lowest-order case.
We cite here, in particular, the two-dimensional Discrete Duality Finite Volume schemes studied in~\cite{Andreianov.Boyer.ea:07} (cf. also the precursor papers~\cite{Andreianov.Boyer.ea:04,Andreianov.Boyer.ea:05,Andreianov.Boyer.ea:06}), the Mixed Finite Volume scheme of~\cite{Droniou:06} (inspired by~\cite{Droniou.Eymard:06}) valid in arbitrary space dimension, and the Mimetic Finite Difference method of~\cite{Antonietti.Bigoni.ea:14} for $p\in (1,2)$ and under more restrictive assumptions than \eqref{assum:gen}.
High-order discontinuous Galerkin approximations have also been considered in~\cite{Burman.Ern:08}.

\smallskip

The starting point for the present work is the HHO method of~\cite{Di-Pietro.Ern.ea:14}.
In the lowest-order case, it has been shown in \cite[Section 2.5]{Di-Pietro.Ern.ea:14}
that this method belongs to the Hybrid Mixed Mimetic family \cite{Droniou.Eymard.ea:10}, which includes the mixed-hybrid Mimetic Finite Differences
\cite{Brezzi.Lipnikov.et.al:05}, the Hybrid Finite Volume \cite{Eymard.Gallouet.ea:10}
and the Mixed Finite Volume \cite{Droniou.Eymard:06}. The HHO method can therefore be
seen as a higher order version of these schemes.
The (hybrid) degrees of freedom (DOFs) for the HHO method are fully discontinuous polynomials of degree $k\ge 0$ at mesh elements and faces.
The construction hinges on two key ingredients built element-wise:%
\begin{inparaenum}[(i)]
\item a discrete gradient defined from element- and face-based DOFs;
\item a high-order penalty term which vanishes whenever one of its arguments is a polynomial of degree $\le (k+1)$ inside the element.
\end{inparaenum}
These ingredients are combined to build a local contribution, which is then assembled element-wise.
A key feature reducing the computational cost is that only face-based DOFs are globally coupled, whereas element-based DOFs can be locally eliminated by a standard static condensation procedure.

\smallskip

The design of a HHO method for the nonlinear problem~\eqref{eq:strong} entails several new ideas.
A first difference with respect to the linear case is that a more natural choice is to seek the gradient reconstruction in the full space of vector-valued polynomials of degree $\le k$ (as opposed to the space spanned by gradients of scalar-valued polynomials of degree $\le (k+1)$).
The main consequence of this choice is that, when applied to the interpolates of smooth functions, the discrete gradient operator commutes with the $L^2$-projector, and therefore enjoys $L^p$-stability properties (see below).
A second important point is the design of a high-order stabilization term with appropriate scaling.
Here, we propose a generalization of the stabilization term of~\cite{Di-Pietro.Ern.ea:14} which preserves the property of vanishing whenever one of its arguments is a polynomial of degree $\le(k+1)$.
As in the linear case, the construction hinges on the solution of small local linear problems inside each elements, and the possibility of statically condense element-based DOFs remains available.

\smallskip

The convergence analysis is carried out using a compactness argument in the spirit of~\cite{Minty:63}.
This technique, while not delivering an estimate of the convergence rate, has the crucial advantage of relying solely on the solution regularity inherent to the weak formulation. This point is particularly relevant for nonlinear problems, where additional regularity assumptions may turn out to be fictitious.
The theoretical study of the convergence rate for smooth solutions is postponed to a future work.

Adapting the compactness argument has prompted us to develop discrete functional analysis tools whose interest goes beyond the specific method and problem considered in this work.
A first notable set of results are (direct and) reverse Lebesgue and Sobolev embeddings on local polynomial spaces (e.g., on mesh elements and faces, but curved geometries are also allowed).
The term reverse refers to the fact that the \emph{largest} exponent (semi-)norm is bounded above by the \emph{lowest} exponent (semi-)norm.
Direct Sobolev embedding for broken spaces on fairly general polytopal meshes are proved in~\cite{Buffa.Ortner:09,Di-Pietro.Ern:10}; specific instances had already been established in~\cite{Arnold:82,Karakashian.Jureidini:98,Brenner:03,Lasis.Suli:03,Girault.Riviere.ea:05}.
Reverse embeddings, on the other hand, are established in \cite[Theorem 4.5.11]{Brenner.Scott:08}, but under the assumption that all mesh elements are affine-equivalent to one (or a finite number of) given fixed reference elements.
This limitation is due to the very generic local finite element spaces considered therein.
Exploiting the fact that we deal with \emph{polynomial} local spaces, we can establish a more general version of reverse inequalities, that does not require to specify any particular geometry of the elements (only their non-degeneracy).
Reverse Lebesgue embeddings are a crucial ingredient to prove the stability of the HHO method.

A second set of results concerns the stability and approximation properties of the $L^2$-projector on local polynomial spaces.
More specifically, we prove under very general geometric assumptions that the $L^2$-projector is $L^p$-stable for any index $p\in[1,+\infty]$, and that it has optimal approximation properties in local polynomial spaces.
Stability results for (global) projectors onto finite element spaces can be found in
\cite{Crouzeix.Thomee:87,Carstensen:02,Bramble.Pasciak.et.al:02,Bank.Yserentant:14}. However, these
references mostly consider $H^1$-stability, and assume quite restrictive (and sometimes
difficult to check) geometrical assumptions on the meshes.
These limitations are a consequence of dealing with projectors on global finite element spaces, that include some form of continuity property between the mesh elements.
On discontinuous polynomial spaces such as the ones used in HHO methods, we can establish more general $L^p$- and $W^{s,p}$-stability and approximation properties of local $L^2$-projectors.
The approximation results extend to the $W^{s,p}$-setting the ones in~\cite[Section~1.4.4]{Di-Pietro.Ern:12}, based in turn on the ideas of~\cite{Dupont.Scott:80}.

Finally, a third set of discrete functional analysis tools are specific to polynomial spaces with a hybrid structure, i.e., using as DOFs polynomials at elements and faces.
In this case, building on the results of~\cite{Di-Pietro.Ern:10} for discontinuous Galerkin methods
(inspired by the low-order discrete functional analysis results of \cite{Droniou.Eymard:06,Eymard.Gallouet.ea:10}), we introduce a suitable discrete $W^{1,p}$-like norm and prove a discrete counterpart of Sobolev embeddings and a compactness result for the discrete gradient reconstruction upon which the HHO method hinges.

\smallskip

The material is organized as follows:
in Section~\ref{sec:cont.setting} we recall a set of standard assumptions to write a weak formulation for problem~\eqref{eq:strong}; 
in Section~\ref{sec:setting} we detail the discrete setting by specifying the assumptions on the mesh and recalling the basic results on local polynomial spaces;
in Section~\ref{sec:method} we formulate the HHO method, state (without proof) the main stability and convergence results, and provide a few numerical examples;
Section~\ref{sec:dfa} collects the discrete functional analysis tools on hybrid polynomial spaces, which are used in Section~\ref{sec:conv.anal} to prove the stability and convergence of the HHO method;
in Section~\ref{sec:other.bcs} we briefly address the treatment of other boundary conditions and hint at the modifications required in the analysis;
a conclusion is given in Section \ref{sec:ccl} and, finally, in Appendix~\ref{sec:appen} we provide the proofs of the discrete functional analysis results on local polynomial spaces.


\section{Continuous setting}\label{sec:cont.setting}

In this section we detail the assumptions on the function $\bfa$ and write a weak formulation for problem~\eqref{eq:strong}.
Let $p\in(1,+\infty)$ be given, and denote by $p'\eqbydef\frac{p}{p-1}$ the dual exponent of $p$, and by $p^*$ the Sobolev exponent of $p$ such that
\begin{equation}\label{eq:p*}
  p^* = \begin{cases}
    \frac{dp}{d-p} & \text{if $p<d$},
    \\
    +\infty & \text{if $p\ge d$.}
  \end{cases}
\end{equation}
We assume that
\begin{subequations}
	\label{assum:gen}
	\begin{equation}
		\label{hyp:acarat}
		\mbox{$\bfa:\Omega\times\Real\times\Real^d\to\Real^d$ is a Caratheodory function},
	\end{equation}
	\begin{equation}
		\label{hyp:ag}
		\begin{array}{l}
		\exists \gra\in L^{p'}(\Omega)\,,\;\exists \upa\in (0,+\infty)\,,\;\exists r<\frac{p^*}{p'}\,:\,
		|\bfa(\vec{x},s,\vec{\xi})|\le \gra(\vec{x})+ \upa|s|^r+\upa|\vec{\xi}|^{p-1}\\
		\qquad\mbox{for a.e. $\vec{x}\in\Omega$, for all $(s,\vec{\xi})\in\Real\times\Real^d$},
		\end{array}
	\end{equation}
	\begin{equation}
		\label{hyp:am}
		[\bfa(\vec{x},s,\vec{\xi})-\bfa(\vec{x},s,\vec{\eta})]\cdot[\vec{\xi}-\vec{\eta}]\ge 0
		\mbox{ for a.e. $\vec{x}\in\Omega$, for all $(s,\vec{\xi},\vec{\eta})\in\Real\times\Real^d\times\Real^d$},
	\end{equation}
	\begin{equation}
		\label{hyp:ac}
		\exists \coera\in (0,+\infty)\,:\,\bfa(\vec{x},s,\vec{\xi})\cdot\vec{\xi} \ge \coera|\vec{\xi}|^p
		\mbox{ for a.e. $\vec{x}\in\Omega$, for all $(s,\vec{\xi})\in\Real\times\Real^d$},
	\end{equation}
	\begin{equation}
		\label{hyp:fg}
		f\in L^{p'}(\Omega).
	\end{equation}
\end{subequations}
Here, Carathedory function means that $\bfa(\vec{x},\cdot,\cdot)$ is continuous on $\Real\times \Real^d$
for a.e. $\vec{x}\in\Omega$, and $\bfa(\cdot,s,\vec{\xi})$ is measurable on $\Omega$ for all $(s,\vec{\xi})\in \Real\times
\Real^d$.
The Euclidean dot product and norm in $\Real^d$ are denoted by $\vec{x}\cdot\vec{y}$ and $|\vec{x}|$, respectively.
Classically~\cite{Leray.Lions:65}, the weak formulation for \eqref{eq:strong} is
\begin{equation}
	\label{eq:weak}
	\begin{array}{l}
	\mbox{Find }u\in W^{1,p}_0(\Omega)\mbox{ such that, for all $v\in W^{1,p}_0(\Omega)$},\\
	\displaystyle \int_\Omega \bfa(\vec{x},u(\vec{x}),\GRAD u(\vec{x}))\cdot
\GRAD v(\vec{x})\dx=\int_\Omega f(\vec{x})v(\vec{x})\dx.
	\end{array}
\end{equation}

The $p$-Laplace equation is probably the simplest type of Leray-Lions operator, and consists in setting
  \begin{equation}\label{eq:bfa.plap}
    \bfa(\vec{x},u,\GRAD u)=|\GRAD u|^{p-2}\GRAD u.
  \end{equation}
In \cite{Blatter.95}, a simplified model of the stationary motion of glaciers is given by \eqref{eq:strong}
with $$\bfa(\vec{x},u,\GRAD u)=F(|\GRAD u|)\GRAD u,$$ where $F$ is the solution to the implicit
equation $F(s)^{-1}=(sF(s))^{\frac{\alpha}{1-\alpha}}+T_0^{\frac{\alpha}{1-\alpha}}$;
here, $\alpha=2-p\in (0,1)$, $T_0>0$, and the unknown $u$ in \eqref{eq:strong:PDE}
is the horizontal velocity of the ice. It is proved in \cite{Glowinski.Rappaz:03}
that this choice of $\bfa$ satisfies \eqref{assum:gen}. We refer the reader
to \cite{Diaz.Thelin:94} for a discussion of models of turbulent flows using
time-dependent versions of \eqref{eq:strong:PDE} with $\bfa$ of the form
$$\bfa(\vec{x},u,\GRAD u)=|\GRAD u - \vec{h}(u)|^{p-2}(\GRAD u - \vec{h}(u))$$ for some function
$\vec{h}:\Real\to\Real^d$.

Existence of a solution to \eqref{eq:weak} is a consequence of the general 
results in \cite{Leray.Lions:65}. Even if $\bfa$ does not depend on $s$, the solution (whether weak
or strong) is usually not unique, see e.g. \cite[Remark 3.4]{Doniou.Eymard.et.al:12}. 
Establishing a uniqueness result on \eqref{eq:weak} requires to strengthen the monotonicity assumption \eqref{hyp:am}.
If $\bfa$ does not depend on $s$ and is \emph{strictly monotone}, in the
sense that \eqref{hyp:am} holds with a strict inequality whenever $\vec{\xi}\not=\vec{\eta}$,
then the uniqueness of the solution to \eqref{eq:weak} is easy to see. Indeed, starting from two solutions $u$ and $u'$,
subtracting the equations and taking $v=u-u'$, we find 
\[
\int_\Omega \left[\bfa(\vec{x},\GRAD u(\vec{x}))-\bfa(\vec{x},\GRAD u'(\vec{x}))\right]
\cdot \left[\GRAD u(\vec{x})-\GRAD u'(\vec{x})\right]\dx=0.
\]
Since the integrand is non-negative, and strictly positive if $\GRAD u(\vec{x})\neq\GRAD u'(\vec{x})$,
this relation shows that $\GRAD u=\GRAD u'$ a.e. on $\Omega$. We then deduce from the homogeneous boundary condition
that $u=u'$ a.e. on $\Omega$.
If $\bfa$ depends on $s$, the uniqueness of the solution is obtained by strengthening
even more the monotonicity assumption \eqref{hyp:am}, and by assuming that $\bfa$ is Lipschitz continuous with respect to
$s$, see \cite{Boccardo.et.al:92,CasadoDiaz.et.al07}.


\section{Discrete setting}\label{sec:setting}

This section presents the discrete setting: admissible mesh sequences, analysis tools on such meshes, DOFs, reduction maps, and reconstruction operators.

\subsection{Assumptions on the mesh}\label{sec:setting:mesh}

Denote by ${\cal H}\subset \Real_*^+ $ a countable set of meshsizes having $0$ as its unique accumulation point.
Following~\cite[Chapter~4]{Di-Pietro.Ern:12}, we consider $h$-refined mesh sequences $(\Th)_{h \in {\cal H}}$ where, for all $ h \in {\cal H} $, $\Th$ is a finite collection of nonempty disjoint open polyhedral elements $T$
such that $\closure{\Omega}=\bigcup_{T\in\Th}\closure{T}$ and $h=\max_{T\in\Th} h_T$
with $h_T$ standing for the diameter of the element $T$.
A face $F$ is defined as a hyperplanar closed connected subset of $\closure{\Omega}$ with positive $ (d{-}1) $-dimensional Hausdorff measure and such that
\begin{inparaenum}[(i)]%
  \item either there exist $T_1,T_2\in\Th $ such that $F\subset\partial T_1\cap\partial T_2$ and $F$ is called an interface or 
  \item there exists $T\in\Th$ such that $F\subset\partial T\cap\partial\Omega$ and $F$ is called a boundary face.
\end{inparaenum}
Interfaces are collected in the set $\Fhi$, boundary faces in $\Fhb$, and we let $\Fh\eqbydef\Fhi\cup\Fhb$.
The diameter of a face $F\in\Fh$ is denoted by $h_F$.
For all $T\in\Th$, $\Fh[T]\eqbydef\{F\in\Fh\st F\subset\partial T\}$ denotes the set of faces contained in $\partial T$ (with $\partial T$ denoting the boundary of $T$) and, for all $F\in\Fh[T]$, $\normal_{TF}$ is the unit normal to $F$ pointing out of $T$.
Symmetrically, for all $F\in\Fh$, we let $\Th[F]\eqbydef\{T\in\Th\st F\subset \partial T\}$
the set of elements having $F$ as a face.

Our analysis hinges on the following assumption on the mesh sequence. 

\begin{assumption}[Admissible mesh sequence]
  \label{def:adm.Th}
  For all $h\in{\cal H}$, $\Th$ admits a matching simplicial submesh $\fTh$ and there exists a real number $\varrho>0$ such that, for all $h\in{\cal H}$:
  \begin{inparaenum}[(i)]%
  \item for all simplex $S\in\fTh$ of diameter $h_S$ and inradius $r_S$, $\varrho h_S\le r_S$, and
  \item for all $T\in\Th$, and all $S\in\fTh$ such that $S\subset T$, $\varrho h_T \le h_S$.
  \end{inparaenum}
\end{assumption}

The simplicial submesh in this assumption is just a theoretical tool, and it is not used in the actual construction of the discretization method.
Given an admissible mesh sequence, for all $h\in{\cal H}$, all $T\in\Th$, and all $F\in\Fh[T]$, $h_F$ is uniformly comparable to $h_T$ in the sense that (cf.~\cite[Lemma~1.42]{Di-Pietro.Ern:12}):
\begin{equation}
  \label{eq:hF}
  \varrho^2 h_T\leq h_F\leq h_T.
\end{equation}
Moreover, \cite[Lemma 1.41]{Di-Pietro.Ern:12} shows that there exists an integer $\Np$ depending on $\varrho$ such that
\begin{equation}
  \label{eq:Np}
  \forall h\in{\cal H}\,:\,
  \max_{T\in\Th}\card{\Fh[T]}\le\Np.
\end{equation}
Finally, by~\cite[Lemma 1.40]{Di-Pietro.Ern:12}, there is an integer $N_{\rm s}$ depending on $\varrho$ such that
  \begin{equation}
    \label{eq:Ns}
    \forall h\in{\cal H}\,:\,
    \max_{T\in\Th}\card{\{S\in\fTh\st S\subset T\}}\le N_{\rm s}.
  \end{equation}

\subsection{Basic results on local polynomial spaces}\label{sec:local.poly}
The building blocks for the HHO method are local polynomial spaces on elements and faces.
Let an integer $l\ge 0$ be fixed. Let $U$ be a subset of $\Real^{N}$ (for some $N\ge 1$), $H_U$ the affine space spanned by $U$, $d_U$ its dimension, and assume that $U$ has a non-empty interior in $H_U$.
We denote by $\Poly{l}(U)$ the space spanned by $d_U$-variate polynomials on $H_U$ of total degree $\le l$.
In the following sections, we will typically have $N=d$ and the set $U$ will represent a mesh element (and $d_U=d$) or a mesh face (and $d_U=d-1$).
We note, in passing, that a subset $U$ with curved boundaries is also allowed except in Lemma~\ref{lem:Wkp.interp.trace}, which is why we use the different notation $T$ instead of $U$ in this lemma.

A key element in the construction are $L^2$-projectors onto local polynomial spaces on bounded subsets $U\subset\Real^{N}$.
The $L^2$-projector $\lproj[U]{l}: L^1(U)\to \Poly[N]{l}(U)$ is defined as follows:
For any $w\in L^1(U)$, $\lproj[U]{l}w$ is the unique element of $\Poly{l}(U)$ such that
\begin{equation}\label{def:lproj}
  \forall v\in\Poly{l}(U)\,:\,
  \int_U \lproj[U]{l}w(\vec{x})v(\vec{x}) d\vec{x}=
  \int_U w(\vec{x})v(\vec{x}) d\vec{x}.
\end{equation}
Note that the regularity $w\in L^1(U)$ suffices to integrate $w$ against polynomials on $U$ (which are bounded functions).
In what follows, we state some stability and approximation properties for the $L^2$-projector.
The proofs are postponed to Appendix~\ref{sec:LPest.L2proj}.

\begin{restatable}[$L^p$-stability of $L^2$-projectors on polynomial spaces]{lemma}{stabproj}\label{lem:stab.proj}
  Let $U$ be a measurable subset of $\Real^N$, with inradius $r_U$ and diameter $h_U$, such that
  \begin{equation}\label{eq:reg.U}
    \frac{r_U}{h_U}\ge \delta>0.
  \end{equation}
  Let $k\in\mathbb{N}$ and $p\in [1,+\infty]$.
  Then, there exists $C$ only depending on $N$, $\delta$, $k$ and $p$
  such that
  \begin{equation}\label{stab.proj}
    \forall g\in L^{p}(U)\,:\,\norm[L^p(U)]{\lproj[U]{k}g}\le C\norm[L^p(U)]{g}.
  \end{equation}
\end{restatable}

\begin{remark}[Geometric regularity~\eqref{eq:reg.U} for mesh elements and faces]\label{rem:geom.reg}
Elements $T\in\Th$ and faces $F\in\Fh$ of an admissible mesh sequence satisfy the geometric 
regularity assumption~\eqref{eq:reg.U} with $\delta=\varrho^2$ and $\delta=\varrho$ respectively.
\end{remark}

In the case where $W^{s,p}(U)$ is continuously embedded in $C(\overline{U})$, the following result can be found in \cite[Theorem 4.4.4]{Brenner.Scott:08}.
This restriction on the space $W^{s,p}(U)$, which would prevent us from analyzing interesting cases for \eqref{eq:strong}, is due to the very general setting chosen for analyzing the interpolation error. Because we focus here on local polynomial spaces and $L^2$-projectors, we can improve this result and obtain optimal interpolation errors for any $s,p$.
If $U$ is an open set of $\Real^N$, $s\in\mathbb{N}$ and $p\in [1,+\infty]$, we recall that $\seminorm[W^{s,p}(U)]{\cdot}$ is
defined by
\[
\forall v\in W^{s,p}(U)\,,\quad\seminorm[W^{s,p}(U)]{v}\eqbydef
\sum_{\alpha\in\mathbb{N}^N,\,|\alpha|_{\ell^1}=s}\norm[L^p(U)]{\partial^\alpha v},
\]
where $|\alpha|_{\ell^1}=\alpha_1+\ldots+\alpha_N$ and $\partial^\alpha=\partial_1^{\alpha_1}\cdots\partial_N^{\alpha_N}$.

\begin{restatable}[$W^{s,p}$-approximation properties of $L^2$-projectors on polynomial spaces]{lemma}{Wkpinterp}\label{lem:Wkp.interp}
  Let $U$ be an open subset of $\Real^N$ with diameter $h_U$,
  such that $U$ is star-shaped with respect to a ball of radius $\rho h_U$ for
  some $\rho>0$. Let $k\in\mathbb{N}$, $s\in\{1,\ldots,k+1\}$ and $p\in [1,+\infty]$.
  Then, there exists $C$ only depending on $N$, $\rho$, $k$, $s$ and $p$ such that
  \begin{equation}\label{eq:approx.lproj.Wsp}
    \forall m\in \{0,\ldots,s\}\,,\;\forall v\in W^{s,p}(U)\,:\,
    \seminorm[W^{m,p}(U)]{v-\lproj[U]{k}v}\le Ch_U^{s-m}\seminorm[W^{s,p}(U)]{v}.
  \end{equation}
\end{restatable}

\begin{remark}\label{rem:union.ss} Using \cite[Section 7]{Dupont.Scott:80}, the result still holds
  if $U$ is a finite union of domains that are star-shaped with respect to balls
  of radius comparable to $h_U$. This enables us to use Lemma~\ref{lem:Wkp.interp} on elements
  of admissible mesh sequences, which are the union of a finite number of simplices; cf.~\eqref{eq:Ns}.
\end{remark}

The next result estimates the trace of the error, and therefore requires more
geometric assumptions on the domain (which, in the following sections, will be invariably a mesh element $T$).

\begin{restatable}[Approximation properties of traces of $L^2$-projectors on polynomial spaces]{lemma}{Wkpinterptrace}
  \label{lem:Wkp.interp.trace}
  Let $T$ be a polyhedral subset of $\Real^N$ with diameter $h_T$,
  such that $T$ is the union of disjoint simplices $S$ of diameter $h_S$
  and inradius $r_S$ such that $\varrho^2 h_T\le \varrho h_S\le r_S$ for some $\varrho>0$.
  Let $k\in\mathbb{N}$, $s\in\{1,\ldots,k+1\}$ and $p\in [1,+\infty]$.
  Then, there exists $C$ only depending on $N$, $\varrho$, $k$, $s$ and $p$ such that
  \begin{equation}\label{eq:approx.lproj.Wsp.trace}
    \forall m\in \{0,\ldots,s-1\}\,,\;\forall v\in W^{s,p}(T)\,:\,
    h_T^{\frac1p}\seminorm[{W^{m,p}(\Fh[T])}]{v-\lproj[T]{k}v}\le Ch_T^{s-m}\seminorm[W^{s,p}(T)]{v}.
  \end{equation}
  Here, $W^{m,p}(\Fh[T])$ is the set of functions that belong to $W^{m,p}(F)$ for
  any hyperplanar face $F$ of $T$, with corresponding broken norm.
\end{restatable}

Finally, the triangle inequality applied to \eqref{eq:approx.lproj.Wsp} (with $m=s$)
and to \eqref{eq:approx.lproj.Wsp.trace} (with $m=s-1$)
immediately gives the following extension of Lemma \ref{lem:stab.proj}.

\begin{corollary}[$W^{s,p}$-stability of $L^2$-projectors on polynomial spaces]\label{cor:Wsp.stab}
  The following holds:
  \begin{enumerate}[(i)]
  \item Under the assumptions of Lemma \ref{lem:Wkp.interp}, we have, with $C$ only depending on $N$, $\rho$, $k$, $s$ and $p$,
    \[
    \forall v\in W^{s,p}(U)\,:\,
    \seminorm[W^{s,p}(U)]{\lproj[U]{k}v}\le C\seminorm[W^{s,p}(U)]{v};
    \]
  \item Under the assumptions of Lemma \ref{lem:Wkp.interp.trace}, we have with $C$ only depending on $N$, $\varrho$, $k$, $s$ and $p$,
    \[
    \forall v\in W^{s,p}(T)\,:\,
    \seminorm[{W^{s-1,p}(\Fh[T])}]{\lproj[T]{k}v}\le Ch_T^{\frac{1}{p'}}\seminorm[W^{s,p}(T)]{v}
    +\seminorm[{W^{s-1,p}(\Fh[T])}]{v}.
    \]
  \end{enumerate}
\end{corollary}


\section{The Hybrid High-Order method}\label{sec:method}

In this section we introduce the space of degrees of freedom, define the gradient and potential reconstructions at the heart of the HHO method, state the discrete problem along with the main stability and convergence results, and provide some numerical examples.

\subsection{Local degrees of freedom, interpolation and reconstructions}
\label{sec:DOFs}
Let a polynomial degree $k\ge 0$ and an element $T\in\Th$ be fixed.
We define the local space of DOFs
\begin{equation}
  \label{eq:UT}
    \UT\eqbydef\Poly{k}(T)\times\left(
    \bigtimes_{F\in\Fh[T]}\Poly[d-1]{k}(F)
    \right),%
\end{equation}
cf. Figure~\ref{fig:DOFs}, and we use the underline notation $\sv=(\unv[T],(\unv[F])_{F\in\Fh[T]})$ for a generic element $\sv\in\UT$.
\begin{figure}
\centering	
  \includegraphics[height=3cm]{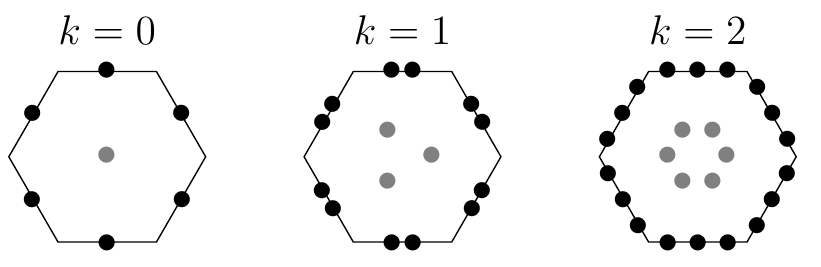}
  \caption{Degrees of freedom for $k\in\{0,1,2\}$. Shaded DOFs can be locally eliminated by static condensation.\label{fig:DOFs}}
\end{figure}
We define the local interpolation operator $\IT: W^{1,1}(T)\to\UT$ such that, for all $v\in W^{1,1}(T)$,
\begin{equation}
  \label{eq:IT}
  \IT v \eqbydef \left(\lproj[T]{k}v, (\lproj[F]{k}v)_{F\in\Fh[T]}\right).
\end{equation}

\begin{remark}[Domain for the interpolation operator]\label{rem:domain.interp}
The local interpolation operator is well-defined for functions $v\in W^{1,1}(T)$ since $v$ is clearly in $L^{1}(T)$, the domain of $\lproj[T]{k}$, and its trace on every face $F\in\Fh[T]$ is in $L^1(F)$, the domain of $\lproj[F]{k}$.
  In passing, in our convergence proofs we only need apply the interpolation operator to classically regular functions; cf., in particular, the proof of Theorem~\ref{th:conv.hho} given in Section~\ref{sec:conv.anal}.
  \end{remark}
  
Based on the local DOFs, we introduce reconstructions of the gradient and of the potential that will be instrumental in the formulation of the method.
In what follows, $(\cdot,\cdot)_T$ and $(\cdot,\cdot)_F$ denote the $L^2$-inner products on $T$ and $F$,
respectively. The same notation is used in the vector case $(L^2)^d$.
We define the local discrete gradient operator $\GT:\UT\to \Poly{k}(T)^d$ such that,
if $\sv\eqbydef(\unv[T],(\unv[F])_{F\in\Fh[T]})\in\UT$, then for all $\bphi\in \Poly{k}(T)^d$,
\begin{subequations}
  \label{eq:GT}
  \begin{align}
    \label{def:GT}    
    (\GT \sv,\bphi)_T
    &= (\GRAD \unv[T],\bphi)_T
    + \sum_{F\in\Fh[T]} (\unv[F]-\unv[T],\bphi\SCAL\normal_{TF})_F
    \\
    \label{def:GT.bis}
    &= -(\unv[T],\DIV\bphi)_T
      + \sum_{F\in\Fh[T]} (\unv[F], \bphi\SCAL\normal_{TF})_F.
  \end{align}
\end{subequations}
Recalling the definition~\eqref{eq:IT} of $\IT$, and using~\eqref{def:GT.bis} together with the definition~\eqref{def:lproj} of the $L^{2}$-projector, one can prove that the following commuting property holds: For all $v\in W^{1,1}(T)$,
\begin{equation}\label{eq:commut.GT}
  \GT\IT v = \lproj[T]{k} (\GRAD v),
\end{equation}
where $\lproj[T]{k}$ acts component-wise.
As a result, by~\eqref{eq:approx.lproj.Wsp} and \eqref{eq:approx.lproj.Wsp.trace}, $\GT\IT$ has optimal approximation properties in $\Poly{k}(T)^d$.
The local potential reconstruction operator $\pT:\UT\to\Poly{k+1}(T)$ is such that, for all $\sv\in\UT$, the gradient of $\pT\sv$ is
the orthogonal projection on $\nabla\Poly{k+1}(T)$ of $\GT\sv$, and the average of $\pT\sv$ over $T$ coincides with the average of $\unv[T]$,
  \begin{equation}\label{eq:pT}
    (\GRAD\pT\sv-\GT\sv,\GRAD w)_T=0\quad\forall w\in\Poly{k+1}(T)\text{\quad and\quad}
    \int_T (\pT\sv(\vec{x})-\unv[T](\vec{x}))d\vec{x} = 0.
\end{equation}
  For all $v\in H^1(T)$, we have the following Euler equation:
\begin{equation}\label{eq:euler.pT}
  (\GRAD(\pT\IT v - v), \GRAD w)_T = 0\qquad\forall w\in\Poly{k+1}(T),
\end{equation}
which shows that $\pT\IT$ is nothing but the usual elliptic projector on $\Poly{k+1}(T)$.

\subsection{Global degrees of freedom, interpolation and reconstructions}
Local DOFs are collected in the following global space obtained by patching interface values:
\begin{equation*}
  \Uh\eqbydef\left(
  \bigtimes_{T\in\Th}\Poly{k}(T)
  \right)\times
  \left(
  \bigtimes_{F\in\Fh}\Poly[d-1]{k}(F)
  \right).
\end{equation*}
We use the notation $\sv[h]=((\unv[T])_{T\in\Th},(\unv[F])_{F\in\Fh})$ for a generic element $\sv[h]\in\Uh$ and, for all $T\in\Th$, it is understood that $\sv[T]=(\unv[T],(\unv[F])_{F\in\Fh[T]})$ denotes the restriction of $\sv[h]$ to $\UT$.
The global interpolation operator $\Ih:W^{1,1}(\Omega)\to\Uh$ is defined such that, for all $v\in W^{1,1}(\Omega)$,
\begin{equation}\label{eq:Ih}
  \Ih v \eqbydef ( (\lproj[T]{k} v)_{T\in\Th}, (\lproj[F]{k} v)_{F\in\Fh} ).
\end{equation}
Interface DOFs are well-defined thanks to the regularity of functions in $W^{1,1}(\Omega)$.
With $\Poly{k}(\Th)$ usual broken polynomial space on $\Th$, for all $\sv[h]\in\Uh$ we denote by $\unv[h]$ the unique function in $\Poly{k}(\Th)$ such that
\begin{equation}\label{def:unvh}
  \restrto{\unv[h]}{T}=\unv[T]\qquad\forall T\in\Th.
\end{equation}
Finally, we introduce the global discrete gradient operator $\Gh:\Uh\to\Poly{k}(\Th)^d$ and potential reconstruction $\ph:\Uh\to\Poly{k+1}(\Th)$ such that, for all $\sv[h]\in\Uh$,
\begin{equation}\label{def:Gh.ph}
  \restrto{(\Gh\sv[h])}{T}=\GT\sv[T]\quad\mbox{ and }\quad \restrto{(\ph\sv[h])}{T}=\pT\sv[T]\qquad\forall T\in\Th.
\end{equation}

\subsection{Discrete problem and main results}

Define the following subspace of $\Uh$ which strongly incorporates the homogeneous Dirichlet boundary condition~\eqref{eq:strong:BC}:
\begin{equation}\label{def:UhD}
  \UhD\eqbydef\left\{
  \sv[h]\in\Uh\st \unv[F]=0\quad\forall F\in\Fhb
  \right\}.
\end{equation}
We consider the following approximation of \eqref{eq:weak}:
\begin{subequations}\label{def:hho.scheme}
\begin{equation}
  \label{eq:hho-glob}
  \mbox{Find $\su[h]\in \UhD$ such that, for any $\sv[h]\in \UhD$, }\asch(\su[h],\sv[h])
  = \int_{\Omega} f(\vec{x}) \unv[h](\vec{x})d\vec{x},
\end{equation}
where $A:\Uh\times\Uh\to\Real$ is assembled element-wise
\begin{equation}
	\label{eq:hho-assembly}
	\asch(\su[h],\sv[h])\eqbydef\sum_{T\in\Th} \asch_T(\su[T],\sv[T]),
\end{equation}
from the local contributions $A_T:\UT\times\UT\to\Real$, $T\in\Th$, defined such that
\begin{equation}
  \label{eq:hho-loc}
  \begin{aligned}
    \asch_T(\su,\sv)\eqbydef{}&\int_T \bfa(\vec{x},\unu[T](\vec{x}),\GT\su(\vec{x}))\cdot\GT\sv(\vec{x})d\vec{x} + s_T(\su,\sv),
    \\
    s_T(\su,\sv)\eqbydef{}&\sum_{F\in\Fh[T]}h_F^{1-p}\int_F 
    \left|\lproj[F]{k}(\unu[F]-\PT\su[T])(\vec{x})\right|^{p-2}
    \lproj[F]{k}(\unu[F]-\PT\su[T])(\vec{x})\lproj[F]{k}(\unv[F]-\PT\sv[T])(\vec{x})ds(\vec{x}),
  \end{aligned}
\end{equation}
with $\PT:\UT\to\Poly{k+1}(T)$ denoting a second potential reconstruction such that, for all $\sv\in\UT$,
\begin{equation}\label{eq:PT}
  \PT\sv\eqbydef\unv[T]+(\pT\sv-\lproj[T]{k}\pT\sv).
\end{equation}
\end{subequations}

\begin{remark} This elaborate expression for the stabilization contribution $s_T$ aims
at preserving the approximation qualities of the consistent contribution in $\asch_T$. As shown by
\eqref{eq:commut.GT}, $\GT$ is exactly the gradient on (interpolations of) polynomials of degree $\le k+1$
inside the element.
To preserve this exactness property in $\asch_T$, the stabilisation term $s_T$ must therefore vanish on (interpolations of) polynomials of degree $\le k+1$ inside the element.
The choice in \eqref{eq:hho-loc} is one option that satisfies this property; other options include penalizing instead of $\lproj[F]{k}(\unv[F]-\PT\sv[T])$ a combination of differences of the form $\lproj[F]{k}(\unv[F]-\pT\sv[T])$ and $\lproj[T]{k}(\unv[T]-\pT\sv[T])$, weighted according the exponent $p$ and their scaling properties with respect to the cell size.

On the contrary, the more naive choice consisting in penalizing the difference $(\unv[F]-\unv[T])$ would only ensure that this stabilisation vanishes on polynomials of degree $\le k$ inside the element. This would prevent, e.g., from attaining the optimal convergence orders proved in~\cite{Di-Pietro.Ern.ea:14} for the linear case with $p=2$.
\end{remark}%

\begin{remark}[Static condensation]\label{rem:stat.cond}
  Problem~\eqref{eq:hho-glob} is a system of nonlinear algebraic equations, which can be solved using an iterative algorithm.
  When first order (Newton-like) algorithms are used, element-based DOFs can be locally eliminated at each iteration by a standard static condensation procedure.
\end{remark}

\begin{remark}[Variants]\label{rem:variations}
  Following~\cite{Cockburn.Di-Pietro.ea:15}, one could replace the space $\UT$ of~\eqref{eq:UT} with
  $$
  \UT[l,k]\eqbydef
  \Poly{l}(T)\times\left\{
  \bigtimes_{F\in\Fh}\Poly[d-1]{k}(F)
  \right\},
  $$
  for $k\ge 0$ and $l\in\{k-1,k,k+1\}$.
  For the sake of simplicity, we only consider here the case $l=k-1$ when $k\ge 1$.
  For $k=0$ and $l=k-1$, some technical modifications (not detailed here) are required owing to the absence of element-based DOFs.
  The local reconstruction operators $\GT$ defined by~\eqref{eq:GT} and $\pT$ defined by~\eqref{eq:pT} still map on $\Poly{k}(T)^d$ and $\Poly{k+1}(T)$, respectively (their domain changes, but we keep the same notation for the sake of simplicity).
  A close inspection shows that both key properties~\eqref{eq:commut.GT} and~\eqref{eq:euler.pT} remain valid for the proposed choices for $l$.
  The second potential reconstruction operator $\PT$ defined by~\eqref{eq:PT}, on the other hand, is replaced by $\PT[l,k+1]:\UT[l,k]\to\Poly{k+1}(T)$ such that, for all $\sv\in\UT[l,k]$, $\PT[l,k+1]\sv\eqbydef\unv[T]+(\pT\sv-\lproj[T]{l}\pT\sv)$.
  The interest of the case $l=k+1$ is that it holds, for all $\sv\in\UT$, $\PT[k+1,k+1]\sv=\unv[T]$, and the stabilization contribution takes the simpler form
  $$
  s_T(\su,\sv)=
  \sum_{F\in\Fh[T]}h_F^{1-p}\int_F \left|\lproj[F]{k}(\unu[F]-\unu[T])(\vec{x})\right|^{p-2}
  \lproj[F]{k}(\unu[F]-\unu[T])(\vec{x})\lproj[F]{k}(\unv[F]-\unv[T])(\vec{x})ds(\vec{x}).
  $$
  This simplification, however, comes at the price of having more element-based DOFs, which leads in turn to more onerous local problems for both the computation of the operator reconstructions and the elimination of element-based unknowns by static condensation.
  We also notice that the choice $l=k+1$ is close in spirit to the Hybridizable Discontinuous Galerkin methods introduced in~\cite{Cockburn.Gopalakrishnan.ea:09} for a linear diffusion problem.
  The choice $l=k-1$, on the other hand, can be related to the High-Order Mimetic method introduced in~\cite{Lipnikov.Manzini:14} in the context of linear elliptic equations.
\end{remark}

We next state our main results for problem~\eqref{def:hho.scheme}.
The proofs are postponed to Section~\ref{sec:conv.anal}.

\begin{theorem}[Existence of a discrete solution]\label{th:exist.hho}
  Under Assumption \eqref{assum:gen}, there exists at least one solution $\su[h]\in\UhD$ to \eqref{def:hho.scheme}.
\end{theorem}

\begin{theorem}[Convergence]\label{th:conv.hho}
  We assume \eqref{assum:gen}, and we let $(\Th)_{h\in\mathcal H}$ be an admissible mesh sequence.
  For all $h\in\mathcal H$, we let $\su[h]\in\UhD$ be a solution to \eqref{def:hho.scheme} on $\Th$.
  Then up to a subsequence as $h\to 0$, recalling the definition~\eqref{eq:p*} of the Sobolev index $p^*$,
  \begin{itemize}
  \item $\unu[h]\to u$ and $\ph\su[h]\to u$ strongly in $L^q(\Omega)$ for all $q<p^*$,
  \item $\Gh\su[h]\to \GRAD u$ weakly in $L^p(\Omega)^d$,
  \end{itemize}
  where $u\in W^{1,p}_0(\Omega)$ solves the weak formulation \eqref{eq:weak} of the PDE \eqref{eq:strong}.
  If we assume, moreover, that $\bfa$ is strictly monotone, that is the inequality
  in \eqref{hyp:am} is strict if $\vec{\xi}\neq\vec{\eta}$, then
  \begin{itemize}
  \item $\Gh\su[h]\to \GRAD u$ strongly in $L^p(\Omega)^d$,
  \end{itemize}
\end{theorem}

\begin{remark}[Uniqueness]
  If $\bfa$ does not depend on $s$ and is strictly monotone, then the solutions to both the continuous problem \eqref{eq:weak} and its discrete counterpart \eqref{def:hho.scheme} are unique (see the discussion in Section \ref{sec:cont.setting}). In that case, the whole sequence of approximate solutions converges to the weak solution of \eqref{eq:strong}.
\end{remark}

\begin{remark}[Other boundary conditions]
  The results stated in Theorems~\ref{th:exist.hho}--\ref{th:conv.hho} are valid also when more general boundary conditions are considered (this is the case, e.g., in the numerical examples below).
  The modifications required to adapt the analysis to non-homogeneous Dirichlet and Neumann boundary conditions are briefly addressed in Section~\ref{sec:other.bcs}.
\end{remark}

\subsection{Numerical examples}\label{sec:num.ex}

To close this section, we provide a few examples to numerically evaluate the convergence properties of the method (a theoretical study of the convergence rates is postponed to a future work).
We consider the $p$-Laplace problem~\eqref{eq:bfa.plap}.
When $p=2$, we recover the usual (linear) Laplace operator, for which optimal convergence rates are proved in~\cite{Di-Pietro.Ern.ea:14}.
We consider the two-dimensional analytical solution originally proposed in~\cite[Section~4]{Andreianov.Boyer.ea:06}, corresponding to
$u(\vec{x})=\exp(x_1+\pi x_2)$ with suitable source term $f$ inferred from~\eqref{eq:strong:PDE}.
The domain is the unit square $\Omega=(0,1)^2$, and non-homogeneous Dirichlet boundary conditions inferred from the expression of $u$ are enforced on its boundary; cf.~\eqref{eq:hho-glob:nhd} for the precise formulation of the method in this case.
We compute the numerical solutions corresponding to polynomial degrees $k=0,\ldots,4$. 
The meshes used are the triangular and Cartesian mesh families 1 and 2 from the FVCA 5 benchmark~\cite{Herbin.Hubert:08}, and the distorted (predominantly) hexagonal mesh family of~\cite[Section~4.2.3]{Di-Pietro.Lemaire:14}; cf. Figure~\ref{fig:meshes}.
\begin{figure}
  \centering
  \includegraphics[height=3cm]{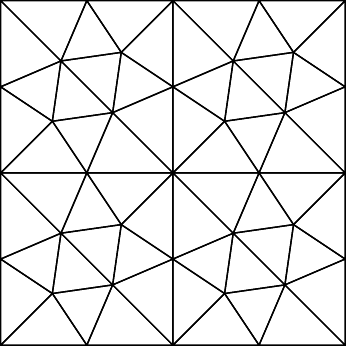}
  \hspace{1em}
  \includegraphics[height=3cm]{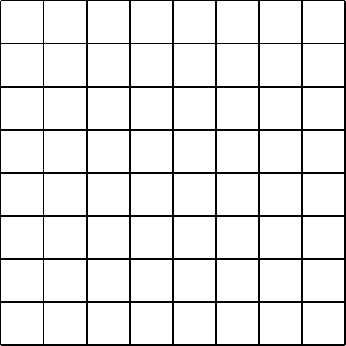}
  \hspace{1em}
  \includegraphics[height=3cm]{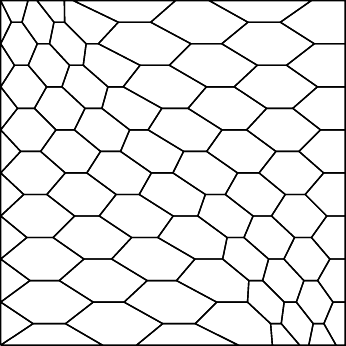}
  \caption{Meshes used in the numerical tests of Section~\ref{sec:num.ex}\label{fig:meshes}.}
\end{figure}

In Figures~\ref{fig:conv.GT:p=3} and~\ref{fig:conv.GT:p=4} we display the convergence of the error $\norm[L^p(\Omega)^d]{\Gh(\su[h]-\Ih u)}$ for $p=3$ and $p=4$, respectively.
In all the cases, we observe that increasing the polynomial degree $k$ improves the convergence rate.
The results obtained in~\cite{Barrett.Liu:94,Andreianov.Boyer.ea:04,Andreianov.Boyer.ea:06} for lowest-order schemes suggest, however, that we should not expect optimal convergence properties in $\Poly{k+1}(\Th)$ except for the linear case $p=2$.
Instead, the order of convergence is expected to depend on both the regularity of the exact solution and the index $p$.
Further numerical tests (not reported here for the sake of brevity) show that the convergence rate improves with $k$ also when considering ``degenerate'' cases (i.e., solutions with a gradient that vanishes in part of the domain, in which case the diffusive properties of \eqref{eq:strong} degenerate), although the gain is, in general, less relevant.
Finally, for the sake of completeness, we report in Figure~\ref{fig:conv.GT:p=4.l=k+1} the numerical results obtained for $p=4$ with the method discussed in Remark~\ref{rem:variations} and corresponding to $l=k+1$.
In this case, taking the element-based DOFs in $\Poly{k+1}(T)$ does not seem to bring any significant advantage in terms of convergence (compare with Figure~\ref{fig:conv.GT:p=4}).

\begin{figure}
  \centering
  \begin{small}
    \tikzexternaldisable
    \ref{conv.legend.p3}
    \tikzexternalenable
  \end{small}
  \\ \vspace{0.25cm}
  \begin{minipage}[b]{0.30\linewidth}
    \begin{tikzpicture}[scale=0.55]
      \begin{loglogaxis}[
          xmin = 0.001,          
          legend columns=-1,
          legend to name=conv.legend.p3,
          legend style={/tikz/every even column/.append style={column sep=0.35cm}}
        ]
        \addplot table[x=meshsize,y={create col/linear regression={y=err_p3}}] {HHOk_plap_0_mesh1.dat}
        coordinate [pos=0.75] (A)
        coordinate [pos=1.00] (B);
        \xdef\slopea{\pgfplotstableregressiona}
        \draw (A) -| (B) node[pos=0.75,anchor=east] {\pgfmathprintnumber{\slopea}};
        \addplot table[x=meshsize,y={create col/linear regression={y=err_p3}}] {HHOk_plap_1_mesh1.dat}
        coordinate [pos=0.75] (A)
        coordinate [pos=1.00] (B);
        \xdef\slopeb{\pgfplotstableregressiona}
        \draw (A) -| (B) node[pos=0.75,anchor=east] {\pgfmathprintnumber{\slopeb}};
        \addplot table[x=meshsize,y={create col/linear regression={y=err_p3}}] {HHOk_plap_2_mesh1.dat}
        coordinate [pos=0.75] (A)
        coordinate [pos=1.00] (B);
        \xdef\slopec{\pgfplotstableregressiona}
        \draw (A) -| (B) node[pos=0.75,anchor=east] {\pgfmathprintnumber{\slopec}};
        \addplot table[x=meshsize,y={create col/linear regression={y=err_p3}}] {HHOk_plap_3_mesh1.dat}
        coordinate [pos=0.75] (A)
        coordinate [pos=1.00] (B);
        \xdef\sloped{\pgfplotstableregressiona}
        \draw (A) -| (B) node[pos=0.75,anchor=east] {\pgfmathprintnumber{\sloped}};
        \addplot table[x=meshsize,y={create col/linear regression={y=err_p3}}] {HHOk_plap_4_mesh1.dat}
        coordinate [pos=0.75] (A)
        coordinate [pos=1.00] (B);
        \xdef\slopee{\pgfplotstableregressiona}
        \draw (A) -| (B) node[pos=0.75,anchor=east] {\pgfmathprintnumber{\slopee}};
        \legend{$k=0$,$k=1$,$k=2$,$k=3$,$k=4$};
      \end{loglogaxis}
    \end{tikzpicture}
    \subcaption{Triangular mesh family}
  \end{minipage}
  \begin{minipage}[b]{0.30\linewidth}
    \begin{tikzpicture}[scale=0.55]
      \begin{loglogaxis}[
          xmin = 0.002,          
          legend style = {
            legend pos = outer north east
          }
        ]
        \addplot table[x=meshsize,y={create col/linear regression={y=err_p3}}] {HHOk_plap_0_mesh2.dat}
        coordinate [pos=0.75] (A)
        coordinate [pos=1.00] (B);
        \xdef\slopea{\pgfplotstableregressiona}
        \draw (A) -| (B) node[pos=0.75,anchor=east] {\pgfmathprintnumber{\slopea}};
        \addplot table[x=meshsize,y={create col/linear regression={y=err_p3}}] {HHOk_plap_1_mesh2.dat}
        coordinate [pos=0.75] (A)
        coordinate [pos=1.00] (B);
        \xdef\slopeb{\pgfplotstableregressiona}
        \draw (A) -| (B) node[pos=0.75,anchor=east] {\pgfmathprintnumber{\slopeb}};
        \addplot table[x=meshsize,y={create col/linear regression={y=err_p3}}] {HHOk_plap_2_mesh2.dat}
        coordinate [pos=0.75] (A)
        coordinate [pos=1.00] (B);
        \xdef\slopec{\pgfplotstableregressiona}
        \draw (A) -| (B) node[pos=0.75,anchor=east] {\pgfmathprintnumber{\slopec}};
        \addplot table[x=meshsize,y={create col/linear regression={y=err_p3}}] {HHOk_plap_3_mesh2.dat}
        coordinate [pos=0.75] (A)
        coordinate [pos=1.00] (B);
        \xdef\sloped{\pgfplotstableregressiona}
        \draw (A) -| (B) node[pos=0.75,anchor=east] {\pgfmathprintnumber{\sloped}};
        \addplot table[x=meshsize,y={create col/linear regression={y=err_p3}}] {HHOk_plap_4_mesh2.dat}
        coordinate [pos=0.75] (A)
        coordinate [pos=1.00] (B);
        \xdef\slopee{\pgfplotstableregressiona}
        \draw (A) -| (B) node[pos=0.75,anchor=east] {\pgfmathprintnumber{\slopee}};
      \end{loglogaxis}
    \end{tikzpicture}
    \subcaption{Cartesian mesh family}
  \end{minipage}
  \begin{minipage}[b]{0.30\linewidth}
    \begin{tikzpicture}[scale=0.55]
      \begin{loglogaxis}[
          xmin = 0.002,          
          legend style = {
            legend pos = outer north east
          }
        ]
        \addplot table[x=meshsize,y={create col/linear regression={y=err_p3}}] {HHOk_plap_0_pi6_tiltedhexagonal.dat}
        coordinate [pos=0.75] (A)
        coordinate [pos=1.00] (B);
        \xdef\slopea{\pgfplotstableregressiona}
        \draw (A) -| (B) node[pos=0.75,anchor=east] {\pgfmathprintnumber{\slopea}};
        \addplot table[x=meshsize,y={create col/linear regression={y=err_p3}}] {HHOk_plap_1_pi6_tiltedhexagonal.dat}
        coordinate [pos=0.75] (A)
        coordinate [pos=1.00] (B);
        \xdef\slopeb{\pgfplotstableregressiona}
        \draw (A) -| (B) node[pos=0.75,anchor=east] {\pgfmathprintnumber{\slopeb}};
        \addplot table[x=meshsize,y={create col/linear regression={y=err_p3}}] {HHOk_plap_2_pi6_tiltedhexagonal.dat}
        coordinate [pos=0.75] (A)
        coordinate [pos=1.00] (B);
        \xdef\slopec{\pgfplotstableregressiona}
        \draw (A) -| (B) node[pos=0.75,anchor=east] {\pgfmathprintnumber{\slopec}};
        \addplot table[x=meshsize,y={create col/linear regression={y=err_p3}}] {HHOk_plap_3_pi6_tiltedhexagonal.dat}
        coordinate [pos=0.75] (A)
        coordinate [pos=1.00] (B);
        \xdef\sloped{\pgfplotstableregressiona}
        \draw (A) -| (B) node[pos=0.75,anchor=east] {\pgfmathprintnumber{\sloped}};
        \addplot table[x=meshsize,y={create col/linear regression={y=err_p3}}] {HHOk_plap_4_pi6_tiltedhexagonal.dat}
        coordinate [pos=0.75] (A)
        coordinate [pos=1.00] (B);
        \xdef\slopee{\pgfplotstableregressiona}
        \draw (A) -| (B) node[pos=0.75,anchor=east] {\pgfmathprintnumber{\slopee}};
      \end{loglogaxis}
    \end{tikzpicture}
    \subcaption{Hexagonal mesh family}
  \end{minipage}  
  \caption{$\norm[L^p(\Omega)^d]{\Gh(\su[h]-\Ih u)}$ vs. $h$, $p=3$.\label{fig:conv.GT:p=3}}
\end{figure}

\begin{figure}
  \centering
  \begin{small}
    \tikzexternaldisable
    \ref{conv.legend.p4}
    \tikzexternalenable
  \end{small}
  \\ \vspace{0.25cm}
  \begin{minipage}[b]{0.30\linewidth}
    \begin{tikzpicture}[scale=0.55]
      \begin{loglogaxis}[
          xmin = 0.001,          
          legend columns=-1,
          legend to name=conv.legend.p4,
          legend style={/tikz/every even column/.append style={column sep=0.35cm}}
        ]
        \addplot table[x=meshsize,y={create col/linear regression={y=err_p4}}] {HHOk_plap_0_mesh1.dat}
        coordinate [pos=0.75] (A)
        coordinate [pos=1.00] (B);
        \xdef\slopea{\pgfplotstableregressiona}
        \draw (A) -| (B) node[pos=0.75,anchor=east] {\pgfmathprintnumber{\slopea}};
        \addplot table[x=meshsize,y={create col/linear regression={y=err_p4}}] {HHOk_plap_1_mesh1.dat}
        coordinate [pos=0.75] (A)
        coordinate [pos=1.00] (B);
        \xdef\slopeb{\pgfplotstableregressiona}
        \draw (A) -| (B) node[pos=0.75,anchor=east] {\pgfmathprintnumber{\slopeb}};
        \addplot table[x=meshsize,y={create col/linear regression={y=err_p4}}] {HHOk_plap_2_mesh1.dat}
        coordinate [pos=0.75] (A)
        coordinate [pos=1.00] (B);
        \xdef\slopec{\pgfplotstableregressiona}
        \draw (A) -| (B) node[pos=0.75,anchor=east] {\pgfmathprintnumber{\slopec}};
        \addplot table[x=meshsize,y={create col/linear regression={y=err_p4}}] {HHOk_plap_3_mesh1.dat}
        coordinate [pos=0.75] (A)
        coordinate [pos=1.00] (B);
        \xdef\sloped{\pgfplotstableregressiona}
        \draw (A) -| (B) node[pos=0.75,anchor=east] {\pgfmathprintnumber{\sloped}};
        \addplot table[x=meshsize,y={create col/linear regression={y=err_p4}}] {HHOk_plap_4_mesh1.dat}
        coordinate [pos=0.75] (A)
        coordinate [pos=1.00] (B);
        \xdef\slopee{\pgfplotstableregressiona}
        \draw (A) -| (B) node[pos=0.75,anchor=east] {\pgfmathprintnumber{\slopee}};
        \legend{$k=0$,$k=1$,$k=2$,$k=3$,$k=4$};
      \end{loglogaxis}
    \end{tikzpicture}
    \subcaption{Triangular mesh family}
  \end{minipage}
  \begin{minipage}[b]{0.30\linewidth}
    \begin{tikzpicture}[scale=0.55]
      \begin{loglogaxis}[
          xmin = 0.002,          
          legend style = {
            legend pos = outer north east
          }
        ]
        \addplot table[x=meshsize,y={create col/linear regression={y=err_p4}}] {HHOk_plap_0_mesh2.dat}
        coordinate [pos=0.75] (A)
        coordinate [pos=1.00] (B);
        \xdef\slopea{\pgfplotstableregressiona}
        \draw (A) -| (B) node[pos=0.75,anchor=east] {\pgfmathprintnumber{\slopea}};
        \addplot table[x=meshsize,y={create col/linear regression={y=err_p4}}] {HHOk_plap_1_mesh2.dat}
        coordinate [pos=0.75] (A)
        coordinate [pos=1.00] (B);
        \xdef\slopeb{\pgfplotstableregressiona}
        \draw (A) -| (B) node[pos=0.75,anchor=east] {\pgfmathprintnumber{\slopeb}};
        \addplot table[x=meshsize,y={create col/linear regression={y=err_p4}}] {HHOk_plap_2_mesh2.dat}
        coordinate [pos=0.75] (A)
        coordinate [pos=1.00] (B);
        \xdef\slopec{\pgfplotstableregressiona}
        \draw (A) -| (B) node[pos=0.75,anchor=east] {\pgfmathprintnumber{\slopec}};
        \addplot table[x=meshsize,y={create col/linear regression={y=err_p4}}] {HHOk_plap_3_mesh2.dat}
        coordinate [pos=0.75] (A)
        coordinate [pos=1.00] (B);
        \xdef\sloped{\pgfplotstableregressiona}
        \draw (A) -| (B) node[pos=0.75,anchor=east] {\pgfmathprintnumber{\sloped}};
        \addplot table[x=meshsize,y={create col/linear regression={y=err_p4}}] {HHOk_plap_4_mesh2.dat}
        coordinate [pos=0.75] (A)
        coordinate [pos=1.00] (B);
        \xdef\slopee{\pgfplotstableregressiona}
        \draw (A) -| (B) node[pos=0.75,anchor=east] {\pgfmathprintnumber{\slopee}};
      \end{loglogaxis}
    \end{tikzpicture}
    \subcaption{Cartesian mesh family}
  \end{minipage}
  \begin{minipage}[b]{0.30\linewidth}
    \begin{tikzpicture}[scale=0.55]
      \begin{loglogaxis}[
          xmin = 0.002,          
          legend style = {
            legend pos = outer north east
          }
        ]
        \addplot table[x=meshsize,y={create col/linear regression={y=err_p4}}] {HHOk_plap_0_pi6_tiltedhexagonal.dat}
        coordinate [pos=0.75] (A)
        coordinate [pos=1.00] (B);
        \xdef\slopea{\pgfplotstableregressiona}
        \draw (A) -| (B) node[pos=0.75,anchor=east] {\pgfmathprintnumber{\slopea}};
        \addplot table[x=meshsize,y={create col/linear regression={y=err_p4}}] {HHOk_plap_1_pi6_tiltedhexagonal.dat}
        coordinate [pos=0.75] (A)
        coordinate [pos=1.00] (B);
        \xdef\slopeb{\pgfplotstableregressiona}
        \draw (A) -| (B) node[pos=0.75,anchor=east] {\pgfmathprintnumber{\slopeb}};
        \addplot table[x=meshsize,y={create col/linear regression={y=err_p4}}] {HHOk_plap_2_pi6_tiltedhexagonal.dat}
        coordinate [pos=0.75] (A)
        coordinate [pos=1.00] (B);
        \xdef\slopec{\pgfplotstableregressiona}
        \draw (A) -| (B) node[pos=0.75,anchor=east] {\pgfmathprintnumber{\slopec}};
        \addplot table[x=meshsize,y={create col/linear regression={y=err_p4}}] {HHOk_plap_3_pi6_tiltedhexagonal.dat}
        coordinate [pos=0.75] (A)
        coordinate [pos=1.00] (B);
        \xdef\sloped{\pgfplotstableregressiona}
        \draw (A) -| (B) node[pos=0.75,anchor=east] {\pgfmathprintnumber{\sloped}};
        \addplot table[x=meshsize,y={create col/linear regression={y=err_p4}}] {HHOk_plap_4_pi6_tiltedhexagonal.dat}
        coordinate [pos=0.75] (A)
        coordinate [pos=1.00] (B);
        \xdef\slopee{\pgfplotstableregressiona}
        \draw (A) -| (B) node[pos=0.75,anchor=east] {\pgfmathprintnumber{\slopee}};
      \end{loglogaxis}
    \end{tikzpicture}
    \subcaption{Hexagonal mesh family}
  \end{minipage}  
  \caption{$\norm[L^p(\Omega)^d]{\Gh(\su[h]-\Ih u)}$ vs. $h$, $p=4$.\label{fig:conv.GT:p=4}}
\end{figure}
\begin{figure}
  \centering
  \begin{small}
    \tikzexternaldisable
    \ref{conv.legend.p3.k+1}
    \tikzexternalenable
  \end{small}
  \\ \vspace{0.25cm}
  \begin{minipage}[b]{0.30\linewidth}
    \begin{tikzpicture}[scale=0.55]
      \begin{loglogaxis}[
          xmin = 0.001,          
          legend columns=-1,
          legend to name=conv.legend.p3.k+1,
          legend style={/tikz/every even column/.append style={column sep=0.35cm}}
        ]
        \addplot table[x=meshsize,y={create col/linear regression={y=err_p4}}] {HHOkpo_plap_0_mesh1.dat}
        coordinate [pos=0.75] (A)
        coordinate [pos=1.00] (B);
        \xdef\slopea{\pgfplotstableregressiona}
        \draw (A) -| (B) node[pos=0.75,anchor=east] {\pgfmathprintnumber{\slopea}};
        \addplot table[x=meshsize,y={create col/linear regression={y=err_p4}}] {HHOkpo_plap_1_mesh1.dat}
        coordinate [pos=0.75] (A)
        coordinate [pos=1.00] (B);
        \xdef\slopeb{\pgfplotstableregressiona}
        \draw (A) -| (B) node[pos=0.75,anchor=east] {\pgfmathprintnumber{\slopeb}};
        \addplot table[x=meshsize,y={create col/linear regression={y=err_p4}}] {HHOkpo_plap_2_mesh1.dat}
        coordinate [pos=0.75] (A)
        coordinate [pos=1.00] (B);
        \xdef\slopec{\pgfplotstableregressiona}
        \draw (A) -| (B) node[pos=0.75,anchor=east] {\pgfmathprintnumber{\slopec}};
        \addplot table[x=meshsize,y={create col/linear regression={y=err_p4}}] {HHOkpo_plap_3_mesh1.dat}
        coordinate [pos=0.75] (A)
        coordinate [pos=1.00] (B);
        \xdef\sloped{\pgfplotstableregressiona}
        \draw (A) -| (B) node[pos=0.75,anchor=east] {\pgfmathprintnumber{\sloped}};
        \addplot table[x=meshsize,y={create col/linear regression={y=err_p4}}] {HHOkpo_plap_4_mesh1.dat}
        coordinate [pos=0.75] (A)
        coordinate [pos=1.00] (B);
        \xdef\slopee{\pgfplotstableregressiona}
        \draw (A) -| (B) node[pos=0.75,anchor=east] {\pgfmathprintnumber{\slopee}};
        \legend{$k=0$,$k=1$,$k=2$,$k=3$,$k=4$};
      \end{loglogaxis}
    \end{tikzpicture}
    \subcaption{Triangular mesh family}
  \end{minipage}
  \begin{minipage}[b]{0.30\linewidth}
    \begin{tikzpicture}[scale=0.55]
      \begin{loglogaxis}[
          xmin = 0.002,          
          legend style = {
            legend pos = outer north east
          }
        ]
        \addplot table[x=meshsize,y={create col/linear regression={y=err_p4}}] {HHOkpo_plap_0_mesh2.dat}
        coordinate [pos=0.75] (A)
        coordinate [pos=1.00] (B);
        \xdef\slopea{\pgfplotstableregressiona}
        \draw (A) -| (B) node[pos=0.75,anchor=east] {\pgfmathprintnumber{\slopea}};
        \addplot table[x=meshsize,y={create col/linear regression={y=err_p4}}] {HHOkpo_plap_1_mesh2.dat}
        coordinate [pos=0.75] (A)
        coordinate [pos=1.00] (B);
        \xdef\slopeb{\pgfplotstableregressiona}
        \draw (A) -| (B) node[pos=0.75,anchor=east] {\pgfmathprintnumber{\slopeb}};
        \addplot table[x=meshsize,y={create col/linear regression={y=err_p4}}] {HHOkpo_plap_2_mesh2.dat}
        coordinate [pos=0.75] (A)
        coordinate [pos=1.00] (B);
        \xdef\slopec{\pgfplotstableregressiona}
        \draw (A) -| (B) node[pos=0.75,anchor=east] {\pgfmathprintnumber{\slopec}};
        \addplot table[x=meshsize,y={create col/linear regression={y=err_p4}}] {HHOkpo_plap_3_mesh2.dat}
        coordinate [pos=0.75] (A)
        coordinate [pos=1.00] (B);
        \xdef\sloped{\pgfplotstableregressiona}
        \draw (A) -| (B) node[pos=0.75,anchor=east] {\pgfmathprintnumber{\sloped}};
        \addplot table[x=meshsize,y={create col/linear regression={y=err_p4}}] {HHOkpo_plap_4_mesh2.dat}
        coordinate [pos=0.75] (A)
        coordinate [pos=1.00] (B);
        \xdef\slopee{\pgfplotstableregressiona}
        \draw (A) -| (B) node[pos=0.75,anchor=east] {\pgfmathprintnumber{\slopee}};
      \end{loglogaxis}
    \end{tikzpicture}
    \subcaption{Cartesian mesh family}
  \end{minipage}
  \begin{minipage}[b]{0.30\linewidth}
    \begin{tikzpicture}[scale=0.55]
      \begin{loglogaxis}[
          xmin = 0.002,          
          legend style = {
            legend pos = outer north east
          }
        ]
        \addplot table[x=meshsize,y={create col/linear regression={y=err_p4}}] {HHOkpo_plap_0_pi6_tiltedhexagonal.dat}
        coordinate [pos=0.75] (A)
        coordinate [pos=1.00] (B);
        \xdef\slopea{\pgfplotstableregressiona}
        \draw (A) -| (B) node[pos=0.75,anchor=east] {\pgfmathprintnumber{\slopea}};
        \addplot table[x=meshsize,y={create col/linear regression={y=err_p4}}] {HHOkpo_plap_1_pi6_tiltedhexagonal.dat}
        coordinate [pos=0.75] (A)
        coordinate [pos=1.00] (B);
        \xdef\slopeb{\pgfplotstableregressiona}
        \draw (A) -| (B) node[pos=0.75,anchor=east] {\pgfmathprintnumber{\slopeb}};
        \addplot table[x=meshsize,y={create col/linear regression={y=err_p4}}] {HHOkpo_plap_2_pi6_tiltedhexagonal.dat}
        coordinate [pos=0.75] (A)
        coordinate [pos=1.00] (B);
        \xdef\slopec{\pgfplotstableregressiona}
        \draw (A) -| (B) node[pos=0.75,anchor=east] {\pgfmathprintnumber{\slopec}};
        \addplot table[x=meshsize,y={create col/linear regression={y=err_p4}}] {HHOkpo_plap_3_pi6_tiltedhexagonal.dat}
        coordinate [pos=0.75] (A)
        coordinate [pos=1.00] (B);
        \xdef\sloped{\pgfplotstableregressiona}
        \draw (A) -| (B) node[pos=0.75,anchor=east] {\pgfmathprintnumber{\sloped}};
        \addplot table[x=meshsize,y={create col/linear regression={y=err_p4}}] {HHOkpo_plap_4_pi6_tiltedhexagonal.dat}
        coordinate [pos=0.75] (A)
        coordinate [pos=1.00] (B);
        \xdef\slopee{\pgfplotstableregressiona}
        \draw (A) -| (B) node[pos=0.75,anchor=east] {\pgfmathprintnumber{\slopee}};
      \end{loglogaxis}
    \end{tikzpicture}
    \subcaption{Hexagonal mesh family}
  \end{minipage}  
  \caption{$\norm[L^p(\Omega)^d]{\Gh(\su[h]-\Ih u)}$ vs. $h$, $p=4$ for the variant of the method discussed in Remark~\ref{rem:variations} and corresponding to $l=k+1$.\label{fig:conv.GT:p=4.l=k+1}}
\end{figure}


\section{Discrete functional analysis tools in hybrid polynomial spaces}\label{sec:dfa}

This section collects discrete functional analysis results on hybrid polynomial spaces that are used in the convergence analysis of Section~\ref{sec:conv.anal}.

\subsection{Discrete $W^{1,p}$-norms}

We introduce the following discrete counterpart of the $W^{1,p}$-seminorm on $\Uh$:
\begin{equation}\label{def:norm.1}
  \norm[1,p,h]{\sv[h]}\eqbydef\left(\sum_{T\in\Th} \norm[1,p,T]{\sv}^p\right)^{\frac1p},
\end{equation}
where the local seminorm $\norm[1,p,T]{{\cdot}}$ on $\UT$ is defined by
\begin{equation}\label{def:norm.1Tp}
  \norm[1,p,T]{\sv}\eqbydef\left(
  \norm[L^p(T)^{d}]{\GRAD \unv[T]}^p
  + \sum_{F\in\Fh[T]}h_F^{1-p}\norm[L^p(F)]{\unv[F]-\unv[T]}^p
  \right)^{\frac1p}.
\end{equation}
It can be checked that the map $\norm[1,p,h]{{\cdot}}$ defines a norm on $\UhD$.
We next show uniform equivalence between the local seminorm defined by~\eqref{def:norm.1Tp} and two local $W^{1,p}$-seminorms defined using the discrete gradient and potential reconstructions (cf.~\eqref{def:GT} and~\eqref{eq:pT}, respectively) and the penalty contribution $s_{T}$ (cf.~\eqref{eq:hho-loc}).
This essentially proves stability for the discrete problem~\eqref{eq:hho-glob} in terms of the $\norm[1,p,h]{{\cdot}}$-norm.
The argument hinges on the following direct and reverse Lebesgue embeddings, whose proof is postponed to Appendix~\ref{sec:est.loc.space}.

\begin{restatable}[Direct and reverse Lebesgue embeddings]{lemma}{lebemb}\label{lem:lebemb}
  Let $U$ be a measurable subset of $\Real^N$ such that \eqref{eq:reg.U} holds.
  Let $k\in\mathbb{N}$ and $q,m\in [1,+\infty]$.
Then, 
\begin{equation}\label{rev:leb}
\forall w\in\Poly[N]{k}(U)\,:\,
\norm[L^q(U)]{w}\approx |U|^{\frac{1}{q}-\frac{1}{m}}
\norm[L^m(U)]{w},
\end{equation}
where $A\approx B$ means that there is a real $M>0$ only depending on $N$, $k$, $\delta$, $q$ and $m$ such that $M^{-1} A\le B\le MA$.
\end{restatable}

We are now ready to prove the norm equivalence.

\begin{lemma}[Equivalence of discrete $W^{1,p}$-seminorms]\label{lem:equivnorms}
  Let $(\Th)_{h\in\mathcal{H}}$ be an admissible mesh sequence and $k\in\mathbb{N}$.
  Let $T\in\Th$, $p\in \lbrack 1,+\infty)$, and denote by $\seminorm[s,p,T]{{\cdot}}$ the local face seminorm such that, for all $\sv\in\UT$, recalling the definition~\eqref{eq:hho-loc} of $s_T$,
  \begin{equation}\label{def:seminorm.spT}
    \seminorm[s,p,T]{\sv}\eqbydef s_{T}(\sv,\sv)^{\frac1{p}}
    =\left(
      \sum_{F\in\Fh[T]}\int_F h_F^{1-p}|\lproj[F]{k}(\unv[F]-\PT\sv)(\vec{x})|^p ds(\vec{x})
    \right)^{\frac1p}.
  \end{equation}
  Then,
  \begin{equation}\label{equiv:norms}
    \norm[1,p,T]{\sv}
    \approx \left( \norm[L^p(T)^d]{\GRAD\pT\sv}^p + \seminorm[s,p,T]{\sv}^p \right)^{\frac1p}
    \approx \left( \norm[L^p(T)^d]{\GT\sv}^p + \seminorm[s,p,T]{\sv}^p \right)^{\frac1p},
  \end{equation}
  where $A\approx B$ means that $M^{-1}A \le B\le M A$ for some real number $M>0$ that may depend on $\Omega$, $\varrho$, $k$ and $p$,
  but does not otherwise depend on the mesh, $T$ or $\sv$.
\end{lemma}

\begin{remark}[Choice of the face seminorm]
  The proof of the norm equivalence does not make use of the specific structure of $s_T$, and could have been proved replacing $\seminorm[s,p,T]{{\cdot}}$ by any other local face seminorm composed by terms scaling on each face $F\in\Fh[T]$ as $h_F^{1-p}\norm[L^p(F)]{{\cdot}}$.
\end{remark}

\begin{proof}
  We abridge $A\lesssim B$ the inequality $A\le MB$ with real $M$ 
  only depending on $\Omega$, $\varrho$, $k$ and $p$.

  \medskip
  
  \emph{Step 1: $p=2$.}
  It was proved in \cite[Lemma~4]{Di-Pietro.Ern.ea:14} that 
  \begin{equation}\label{equiv:norm.2}
    \norm[1,2,T]{\sv}^2\approx\norm[L^2(T)^d]{\GRAD\pT\sv}^2 + \seminorm[s,2,T]{\sv}^2,
  \end{equation}
  which is exactly the first relation in \eqref{equiv:norms} for $p=2$.
  To prove the second, we notice that since, for all $\sv\in\UT$, $\GRAD\pT\sv$ is an orthogonal projection of
  $\GT\sv$ in $L^2(T)^d$, we have $\norm[L^2(T)^d]{\GRAD\pT\sv}\le \norm[L^2(T)^d]{\GT\sv}$. Relation \eqref{equiv:norm.2} therefore shows that
  $$
  \norm[1,2,T]{\sv}^{2}\lesssim \norm[L^2(T)^d]{\GT\sv}^2 + \seminorm[s,2,T]{\sv}^2.
  $$
  To prove the converse estimate,
  we make $\bphi=\GT\sv$ into the definition \eqref{def:GT} of $\GT\sv$, and use
  the Cauchy--Schwarz inequality together with the discrete trace inequality~\cite[Lemma~1.46]{Di-Pietro.Ern:12}  to infer
  $$
  \begin{aligned}
    \norm[L^2(T)^d]{\GT\sv}^2
    &\lesssim
    \norm[L^2(T)^d]{\GRAD\unv[T]}  \norm[L^2(T)^d]{\GT\sv}
    + \!\!\sum_{F\in\Fh[T]} h_F^{-\frac12}\norm[L^2(F)]{\unv[F]-\unv[T]} \norm[L^2(T)^d]{\GT\sv}
    \\
    &\lesssim \norm[1,2,T]{\sv}\norm[L^2(T)^d]{\GT\sv}.
  \end{aligned}
  $$
  This estimate shows that
  $
  \norm[L^2(T)^d]{\GT\sv}\lesssim\norm[1,2,T]{\sv}
  $
  and, combined with~\eqref{equiv:norm.2} to estimate $\seminorm[s,2,T]{\sv}\lesssim\norm[1,2,T]{\sv}$, completes the proof of the case $p=2$.
  
  \medskip

  \emph{Step 2: $p\in [1,+\infty)$.}
    Relation \eqref{equiv:norms} for a generic $p$ can be deduced from the case $p=2$
    thanks to Lemma \ref{lem:lebemb} ($T$ and $F$ clearly satisfy the geometric assumptions therein, cf. Remark~\ref{rem:geom.reg}).
    We only show how to do this to establish
    $$
    \norm[1,p,T]{\sv}^p\lesssim
    \norm[L^p(T)^d]{\GT\sv}^p + \seminorm[s,p,T]{\sv}^p,
    $$
    all the other estimates being obtained in a similar way.
    By admissibility of $(\Th)_{h\in\mathcal{H}}$, we have $h_F|F|\approx |T|$ for any $F\in\Fh[T]$.
    Thus, for $\sv\in\UT$, by Lemma \ref{lem:lebemb},
    \begin{align*}
      \norm[1,p,T]{\sv}^p&\lesssim
      |T|^{1-\frac{p}{2}}\norm[L^2(T)^d]{\GRAD\unv[T]}^p
      +\sum_{F\in\Fh[T]}h_F^{1-p}|F|^{1-\frac{p}{2}}\norm[L^2(F)]{\unv[F]-\unv[T]}^p\\
      &\lesssim |T|^{1-\frac{p}{2}}
      \left(
      \norm[L^2(T)^d]{\GRAD\unv[T]}^2
      + \sum_{F\in\Fh[T]}h_F^{-1}\norm[L^2(F)]{\unv[F]-\unv[T]}^2
      \right)^{\frac{p}{2}},
    \end{align*}
    where, to pass to the second line, we used the inequality 
    \begin{equation}\label{ineq:brute}
      \forall \theta>0,\,\forall a_i\ge 0\,:\,
      \sum_{i=0}^N a_i \le N\left(\sum_{i=1}^N a_i^{\theta}\right)^{\frac{1}{\theta}}
    \end{equation}
    which follows from writing $a_j=(a_j^{\theta})^{\frac{1}{\theta}}\le (\sum_{i=1}^N a_i^{\theta})^{\frac{1}{\theta}}$ for all $j$.
    Apply \eqref{equiv:norms} with $p=2$ and use again
    Lemma \ref{lem:lebemb} and the inequality \eqref{ineq:brute} to infer
    \begin{align*}
      \norm[1,p,T]{\sv}^p&\lesssim |T|^{1-\frac{p}{2}}
      \left(
      \norm[L^2(T)^d]{\GT\sv}^2 +\sum_{F\in\Fh[T]}h_F^{-1}\norm[L^2(F)]{\lproj[F]{k}(\unv[F]-\PT\sv)}^2
      \right)^{\frac{p}{2}}\\
      &\lesssim |T|^{1-\frac{p}{2}}
      \left(
      |T|^{1-\frac{2}{p}}\norm[L^p(T)^d]{\GT\sv}^2 +\sum_{F\in\Fh[T]}h_F^{-1}
      |F|^{1-\frac{2}{p}}\norm[L^p(F)]{\lproj[F]{k}(\unv[F]-\PT\sv)}^2
      \right)^{\frac{p}{2}}\\
      &\lesssim |T|^{1-\frac{p}{2}}
      \left(
      |T|^{1-\frac{2}{p}}
      \bigg(
        \norm[L^p(T)^d]{\GT\sv}^2 +\sum_{F\in\Fh[T]}h_F^{\frac{2}{p}-2}
        \norm[L^p(F)]{\lproj[F]{k}(\unv[F]-\PT\sv)}^2
        \bigg)
      \right)^{\frac{p}{2}}\\
      &\lesssim \norm[L^p(T)^d]{\GT\sv}^p +\sum_{F\in\Fh[T]}h_F^{1-p}
      \norm[L^p(F)]{\lproj[F]{k}(\unv[F]-\PT\sv)}^p.\qedhere
    \end{align*}
\end{proof}

\subsection{Discrete Sobolev embeddings}

The first ingredient of our convergence analysis is the following discrete counterpart of Sobolev embeddings, which will be used in Proposition~\ref{prop:apriori.est} to obtain an a priori estimate of the discrete solution.

\begin{proposition}[Discrete Sobolev embeddings]\label{prop:discsob}
  Let $(\Th)_{h\in\mathcal{H}}$ be an admissible mesh sequence.
  Let $1\le q\le p^*$ if $1\le p<d$ (with $p^*$ defined by~\eqref{eq:p*}) and $1\le q<+\infty$ if $p\ge d$. 
  Then, there exists $C$ only depending on $\Omega$, $\varrho$, $k$, $q$ and $p$ such that
  \begin{equation}\label{eq:discsob}
    \forall\sv[h]\in \UhD\,:\,
    \norm[L^{q}(\Omega)]{\unv[h]}\le C \norm[1,p,h]{\sv[h]}.
  \end{equation}
\end{proposition}

\begin{remark}[Discrete Poincar\'e]
  For $q=p$ (this choice is always possible since $p\le p^{*}$ for any space dimension $d$) this proposition states a discrete Poincar\'e's inequality.
\end{remark}

\begin{proof}
Here, $A\lesssim B$ means that $A\le MB$ for some $M$ only depending on
$\Omega$, $\varrho$, $k$, $q$ and $p$.
We recall the discrete Sobolev embeddings in $\Poly{k}(\Th)$ from \cite[Theorem 5.3]{Di-Pietro.Ern:12} (cf. also~\cite{Buffa.Ortner:09,Di-Pietro.Ern:10}):
\begin{equation}\label{est:sobodG}
  \forall w\in \Poly{k}(\Th)\,:\,
  \norm[L^{q}(\Omega)]{w}\lesssim
\norm[{\rm dG},p]{w},
\end{equation}
where the discrete $W^{1,p}$-norm on $\Poly{k}(\Th)$ is defined by
\begin{equation}\label{def:norm.dG}
  \norm[{\rm dG},p]{w}
  \eqbydef\left(
  \sum_{T\in\Th}\norm[L^p(T)^d]{\GRAD w_T}^p +
  \sum_{F\in\Fh}h_F^{1-p}\norm[L^p(F)]{\jump{w}}^p
  \right)^{\frac1p}.
\end{equation}
Here, for all $T\in\Th$, $w_T\eqbydef\restrto{w}{T}$, while $\jump{w}\eqbydef w_{T_1}-w_{T_2}$ is the jump of $w$ through a face $F\in \Fhi$ such that $\Th[F]=\{T_1,T_2\}$ (the sign is irrelevant).
If $F\in \Fhb$, then $\Th[F]=\{T\}$ and we let $\jump{w}=w_{T}$.
For $\sv[h]\in \UhD$ and $F$ a face between $T_1$ and $T_2$, we have, using the triangle inequality,
\[
\norm[L^p(F)]{\jump{\unv[h]}}\le \norm[L^p(F)]{\unv[T_1]-\unv[F]} + 
\norm[L^p(F)]{\unv[T_2] - \unv[F]}.
\]
Due to the strong boundary conditions, this estimate is also true if $F$ is a boundary
face and the term $T_2$ is removed. Hence, gathering by elements,
$$
\sum_{F\in\Fh}h_F^{1-p}\norm[L^p(F)]{\jump{\unv[h]}}^p
\lesssim \sum_{F\in\Fh}h_F^{1-p}\sum_{T\in\Th[F]}\norm[L^p(F)]{\unv[T]-\unv[F]}^p
=\sum_{T\in\Th}\sum_{F\in\Fh[T]}h_F^{1-p}\norm[L^p(F)]{\unv[T]-\unv[F]}^p
\le \norm[1,p,h]{\sv[h]}^p.
$$
This shows that 
\begin{equation}\label{est:norm.dG}
\norm[{\rm dG},p]{\unv[h]}\lesssim
\norm[1,p,h]{\sv[h]},
\end{equation}
which, plugged into \eqref{est:sobodG}, concludes the proof. \end{proof}

\subsection{Compactness}

The second ingredient for our convergence analysis is the following compactness result for sequences bounded in the $\norm[1,p,h]{{\cdot}}$-norm.

\begin{proposition}[Discrete compactness]\label{prop:comp} Let $(\Th)_{h\in\mathcal{H}}$ be an admissible mesh sequence, and let, \textcolor{red}{for all $h\in\mathcal H$}, $\sv[h]\in \UhD$ be such that $(\norm[1,p,h]{\sv[h]})_{h\in\mathcal{H}}$
is bounded.
Then, there exists $v\in W^{1,p}_0(\Omega)$ such that, up to a subsequence as $h\to 0$, recalling the definition~\eqref{eq:p*} of the Sobolev index $p^*$,
\begin{itemize}
\item $\unv[h]\to v$ and $\ph\sv[h]\to v$ strongly in $L^q(\Omega)$ for all $q<p^*$,
\item $\Gh\sv[h]\to \GRAD v$ weakly in $L^p(\Omega)^d$.
\end{itemize}
\end{proposition}

\begin{remark} If $p^*<+\infty$, the discrete Sobolev embeddings~\eqref{est:sobodG} and Corollary \ref{cor:comp-u-pu} show that both $\unv[h]$ and $\ph\sv[h]$ are bounded in $L^{p^*}(\Omega)$, and their convergence stated in Proposition \ref{prop:comp} extends to $L^{p^*}(\Omega)$-weak.
\end{remark}

The proof of Proposition~\ref{prop:comp} requires an auxiliary result allowing us to compare, for all $\sv[h]\in\Uh$, the broken polynomial function~\eqref{def:unvh} on $\Th$ defined by element DOFs and the potential reconstruction~\eqref{def:Gh.ph}.
Instrumental to obtaining this comparison result is the following Poincar\'e--Wirtinger--Sobolev inequality on broken polynomial spaces, whose interest goes beyond the specific application considered here.

\begin{lemma}[Poincar\'e--Wirtinger--Sobolev inequality for broken polynomial functions with local zero average]\label{lem:poi-wir}
  Let $(\Th)_{h\in\mathcal{H}}$ be an admissible mesh sequence, and let $p\le q\le p^*$ with $p^*$ defined by~\eqref{eq:p*}.
  If $w\in \Poly{k}(\Th)$ satisfies $\int_T w(\vec{x})d\vec{x}=0$ for all $T\in\Th$, then
	there exists $C$ only depending on $\Omega$, $\varrho$, $k$, $q$ and $p$ such that (with $\GRADh$ denoting the usual broken gradient),
  \begin{equation}\label{eq:poi-wir}
    \norm[L^q(\Omega)]{w}
    \le C h^{1+\frac{d}{q}-\frac{d}{p}} \norm[L^p(\Omega)^d]{\GRADh w}.
  \end{equation}
\end{lemma}

\begin{remark} If $p\le d$, the exponent $1+\frac{d}{q}-\frac{d}{p}$ in $h$ is positive
if $q<p^*$ and equal to $0$ if $q=p^*$.
\end{remark}

\begin{proof}
In this proof, $A\lesssim B$ means that $A\le MB$ for some $M$ only depending
on $\Omega$, $\varrho$, $k$, $q$ and $p$.
We have, for all $T\in\Th$, $\lproj[T]{0}w=0$ and therefore, by \eqref{eq:approx.lproj.Wsp} with 
$k=0$, $s=1$ and $m=0$, using Lemma \ref{lem:lebemb} with $m=p$, and recalling that $|T|\lesssim h_T^d$, we write
\begin{equation}\label{PWS.local.1}
\norm[L^q(T)]{w}=\norm[L^q(T)]{w-\lproj[T]{0}w}
\lesssim h_T\norm[L^q(T)^d]{\GRAD w}
\lesssim h_T |T|^{\frac{1}{q}-\frac{1}{p}}
\norm[L^p(T)^d]{\GRAD w}\lesssim h_T^{1+\frac{d}{q}-\frac{d}{p}}
\norm[L^p(T)^d]{\GRAD w}.
\end{equation}
If $q$ is finite, we take the
the power $q$ of this inequality, sum over $T\in\Th$, and use $\norm[L^p(T)^d]{\GRAD w}^{q-p}\le \norm[L^p(\Omega)^d]{\GRADh w}^{q-p}$ (we have $q\ge p$) to infer
\begin{align*}
  \norm[L^q(\Omega)]{w}^q&
  \lesssim  h^{q+d-\frac{dq}{p}}\sum_{T\in\Th}\norm[L^p(T)^d]{\GRAD w}^q
  \le h^{q+d-\frac{dq}{p}}\norm[L^p(\Omega)^d]{\GRADh w}^{q-p}
  \sum_{T\in\Th}\norm[L^p(T)^d]{\GRAD w}^p\\  
  &= h^{q+d-\frac{dq}{p}}\norm[L^p(\Omega)^d]{\GRADh w}^{q-p}
  \norm[L^p(\Omega)^d]{\GRADh w}^p
  = h^{q+d-\frac{dq}{p}}\norm[L^p(\Omega)^d]{\GRADh w}^q.
\end{align*}
Taking the power $1/q$ of this inequality concludes the proof.
If $q=+\infty$, we apply \eqref{PWS.local.1} to $T\in \Th$ such that $\norm[L^\infty(T)]{w}=
\norm[L^\infty(\Omega)]{w}$ to obtain $\norm[L^\infty(\Omega)]{w}
\lesssim h^{1-\frac{d}{p}}\norm[L^p(T)^d]{\GRAD w}
\le h^{1-\frac{d}{p}}\norm[L^p(\Omega)^d]{\GRAD w}$.
\end{proof}

\begin{corollary}[Comparison between $\mathsf{v}_h$ and ${\ph\sv[h]}$]\label{cor:comp-u-pu}
Let $(\Th)_{h\in\mathcal{H}}$ be an admissible mesh sequence, and let $p\le q\le p^*$.
Then, there exists $C$ only depending on $\Omega$, $\varrho$, $k$, $q$ and $p$ such that
\begin{equation}\label{eq:comp-u-pu}
  \forall \sv[h]\in\Uh\,:\,
  \norm[L^{q}(\Omega)]{\unv[h]-\ph\sv[h]}\le C h^{1+\frac{d}{q}-\frac{d}{p}}
  \norm[1,p,h]{\sv[h]}.
\end{equation}
\end{corollary}

\begin{proof}
Here, $A\lesssim B$ means $A\le MB$ for $M$ only depending
on $\Omega$, $\varrho$, $k$, $q$ and $p$.
By the second equation in \eqref{eq:pT}, the average of $\unv[h]-\ph\sv[h]$ over each
element of $\Th$ is zero. Hence,~\eqref{eq:poi-wir} gives
\begin{equation}\label{comp:1}
  \norm[L^{q}(\Omega)]{\unv[h]-\ph\sv[h]}\lesssim h^{1+\frac{d}{q}-\frac{d}{p}}
  \norm[L^p(\Omega)^d]{\GRADh(\unv[h]-\ph\sv[h])}.
\end{equation}
Recalling the definitions~\eqref{def:unvh} of $\unv[h]$ and~\eqref{def:norm.1} of the $\norm[1,p,h]{{\cdot}}$-norm, we have
\begin{equation}\label{comp:2}
  \norm[L^p(\Omega)^d]{\GRADh\unv[h]}^p
  = \sum_{T\in\Th}\norm[L^p(T)^d]{\GRAD\unv[T]}^p
  \le \norm[1,p,h]{\sv[h]}^p.
\end{equation}
Moreover, using the definition~\eqref{def:Gh.ph} of $\ph\sv[h]$ followed by the norm equivalence~\eqref{equiv:norms}, and again the definition~\eqref{def:norm.1} of the $\norm[1,p,h]{{\cdot}}$-norm, it is inferred that
\begin{equation}\label{comp:3}
  \norm[L^p(\Omega)^d]{\GRADh\ph\sv[h]}^p
  =
  \sum_{T\in\Th}\norm[L^p(T)^d]{\GRAD\pT\sv[T]}^p
  \lesssim \sum_{T\in\Th}\norm[1,p,T]{\sv[T]}^p
  = \norm[1,p,h]{\sv[h]}^p.
\end{equation}
We conclude by using the triangle inequality in the right-hand side of \eqref{comp:1} and plugging \eqref{comp:2} and \eqref{comp:3} into the resulting equation. \end{proof}

We are now ready to prove the compactness result stated at the beginning of this section.

\begin{proof}[Proof of Proposition~\ref{prop:comp}]
By \eqref{est:norm.dG}, $(\norm[{\rm dG},p]{\unv[h]})_{h\in\mathcal{H}}$ is bounded. The
discrete Rellich--Kondrachov theorem \cite[Theorem 5.6]{Di-Pietro.Ern:12} ensures that,
up to a subsequence, $\unv[h]$ converges in $L^q(\Omega)$ to some $v$. Since $q<p^*$, Corollary \ref{cor:comp-u-pu}
shows that $\ph\sv[h]$ also converges in this space to the same $v$.

It remains to establish that $v\in W^{1,p}_0(\Omega)$ and that $\Gh\sv[h]$ weakly converges to $\GRAD v$.
To this end, we first notice that $\Gh\sv[h]$ is bounded in $L^p(\Omega)^d$ thanks to the norm equivalence \eqref{equiv:norms}.
Hence, up to a subsequence, it weakly converges in $L^p(\Omega)^d$ to some $\mathcal G$. We take $\bphi\in C^\infty(\Real^d)^d$ and observe that
\begin{align*}
  \int_{\Omega}\Gh\sv[h](\vec{x})\SCAL\bphi(\vec{x})d\vec{x}&=\sum_{T\in\Th}(\GT\sv,\bphi)_T\\
  ={}&\sum_{T\in\Th}(\GT\sv-\GRAD\unv[T],\bphi-\lproj[T]{k}\bphi)_T+
  \sum_{T\in\Th}(\GT\sv-\GRAD\unv[T],\lproj[T]{k}\bphi)_T\\
  &+
  \sum_{T\in\Th}(\GRAD\unv[T],\bphi)_T\\
  ={}&\term_1
  +\sum_{T\in\Th}\sum_{F\in\Fh[T]}(\unv[F]-\unv[T],\lproj[T]{k}\bphi\SCAL
  \normal_{TF})_F+
  \sum_{T\in\Th}(\GRAD\unv[T],\bphi)_T&\mbox{(cf. \eqref{def:GT})}\\
  ={}&\term_1
  +\sum_{T\in\Th}\sum_{F\in\Fh[T]}(\unv[F]-\unv[T],(\lproj[T]{k}\bphi-\bphi)\SCAL
  \normal_{TF})_F-
  \sum_{T\in\Th}(\unv[T],\div\bphi)_T&\mbox{(cf. \eqref{eq:comp.1})}\\
  ={}&\term_1
  +\term_2
  -\int_{\Omega}\unv[h](\vec{x})\div\bphi(\vec{x})d\vec{x}.
\end{align*}
In the penultimate line, we used a element-wise integration by parts,
and the relation 
\begin{equation}\label{eq:comp.1}
  \sum_{T\in\Th}\sum_{F\in\Fh[T]}(\unv[F],\bphi\SCAL\normal_{TF})_F=0,
\end{equation}
which follows from the homogeneous Dirichlet boundary condition incorporated in $\UhD$ (cf.~\eqref{def:UhD}) and from 
$\normal_{T_1F}+\normal_{T_2F}=0$ whenever $F\in \Fhi$ is an interface between the
two elements $T_1$ and $T_2$.
If we prove that, as $h\to 0$, $\term_1+\term_2\to 0$, then we can pass
to the limit and we obtain
\begin{equation}\label{reglim:fin}
\int_\Omega \mathcal G(\vec{x})\cdot\bphi(\vec{x})d\vec{x} = 
-\int_\Omega v(\vec{x})\div\bphi(\vec{x})d\vec{x}.
\end{equation}
Taking $\bphi$ compactly supported in $\Omega$ shows that
$\mathcal G=\GRAD v$, and hence that $v\in W^{1,p}(\Omega)$ and
that $\Gh\sv[h]\to \GRAD v$ weakly in $L^p(\Omega)^d$.
Taking then any $\bphi\in C^\infty(\Real^d)^d$ in \eqref{reglim:fin}
and using an integration by parts shows that the trace of $v$
on $\partial \Omega$ vanishes, which establishes that $v\in W^{1,p}_0(\Omega)$.

It therefore only remains to prove that $\term_1+\term_2\to 0$. In what
follows, $A\lesssim B$ means that $A\le MB$ for some $M$ not depending
on $h$, $\bphi$ or $\sv$.
By Lemma \ref{lem:Wkp.interp} (with $m=0$, $s=1$ and
$p'$ instead of $p$) we have $\norm[L^{p'}(T)^d]{\bphi-\lproj[T]{k}\bphi}\lesssim h
\norm[W^{1,p'}(T)^d]{\bphi}$ and thus
\[
|\term_1|
\lesssim h \left(\sum_{T\in\Th}\norm[L^p(T)^d]{\GT\sv[T]-\GRAD\unv[T]}^p\right)^{1/p}
\norm[W^{1,p'}(\Omega)^d]{\bphi}
\lesssim h \left(
\norm[L^p(\Omega)^d]{\Gh\sv[h]}+\norm[L^p(\Omega)^d]{\GRADh\unv[h]}
\right)\norm[W^{1,p'}(\Omega)^d]{\bphi}.
\]
Since $\norm[1,p,h]{\sv[h]}$ is bounded, the norm equivalence \eqref{equiv:norms} together with the definition~\eqref{def:norm.1} of the $\norm[1,p,h]{{\cdot}}$-norm show that both $\norm[L^p(\Omega)^d]{\Gh\sv[h]}$
and $\norm[L^p(\Omega)^d]{\GRADh\unv[h]}$
remain bounded. Hence, $\term_1\to 0$ as $h\to 0$.
The convergence analysis of $\term_2$ is performed in a similar way.
Using Lemma \ref{lem:Wkp.interp.trace} (with $p'$ instead of $p$)
we have $\norm[L^{p'}(F)]{\bphi-\lproj[T]{k}\bphi}
\lesssim h_T^{\frac{1}{p}}\norm[W^{1,p'}(T)^d]{\bphi}$ and thus,
since $h_T\lesssim h_F$ whenever $F\in \Fh[T]$,
\begin{align*} 
|\term_2|
&\lesssim \sum_{T\in\Th}\sum_{F\in\Fh[T]}h_F^{\frac{1}{p}}\norm[L^p(F)]{\unv[F]-\unv[T]}
\norm[W^{1,p'}(T)^d]{\bphi}
\lesssim \left(\sum_{T\in\Th}\sum_{F\in\Fh[T]}h_F
\norm[L^p(F)]{\unv[F]-\unv[T]}^p\right)^{\frac1p}\norm[W^{1,p'}(\Omega)^d]{\bphi}\\
&\lesssim h\norm[1,p,h]{\sv[h]}\norm[W^{1,p'}(\Omega)^d]{\bphi}.
\end{align*}
The convergence of $\term_2$ to $0$ follows. \end{proof}

\subsection{Strong convergence of the interpolants}

The proof of Theorem~\ref{th:conv.hho} relies on a weak-strong convergence argument.
The last ingredient of the convergence analysis is thus the strong convergence of both the discrete gradient and the stabilization contribution when their argument is the interpolate of a smooth function.
We state here this result in a framework covering more general cases than needed in the proof of Theorem \ref{th:conv.hho} (where the argument of the interpolant is in $C^\infty_c(\Omega)$).
For $r\in\mathbb{N}$ and $q\in [1,+\infty]$, $W^{r,q}(\Th)$ denotes the broken space of
functions $\varphi:\Omega\to \Real$ such that, for any $T\in\Th$, $\varphi|_T\in W^{r,q}(T)$. This space
is endowed with the norm 
\[
\norm[W^{r,q}(\Th)]{\varphi}\eqbydef
	\left\{\begin{array}{ll}
		\displaystyle\left(\sum_{T\in\Th} \norm[W^{r,q}(T)]{\varphi}^q\right)^{1/q}&\mbox{ if $q<+\infty$},\\
		\displaystyle\max_{T\in\Th} \norm[W^{r,q}(T)]{\varphi}&\mbox{ if $q=+\infty$}.
	\end{array}
	\right.
\]

\begin{proposition}[Strong convergence of interpolants]\label{prop:str.cv.interp}
Let $(\Th)_{h\in\mathcal H}$ be an admissible mesh sequence, let $p\in [1,+\infty]$, and let $\Ih$ be defined by \eqref{eq:Ih}.
Then, there exists $C$ not depending on $h$ such that
\begin{equation}\label{eq:interp.error.GT}
  \forall \varphi \in W^{1,1}(\Omega)\cap W^{k+2,p}(\Th)\,:\,
  \norm[L^p(\Omega)]{\Gh\Ih\varphi-\GRAD\varphi}\le C h^{k+1}\norm[W^{k+2,p}(\Th)]{\varphi}.
\end{equation}
As a consequence,
\begin{equation}\label{eq:str.cv.GT}
  \forall \varphi\in W^{1,p}(\Omega)\,:\,
  \Gh\Ih\varphi\to \GRAD\varphi\mbox{ strongly in }L^p(\Omega)^d\mbox{ as $h\to 0$}.
\end{equation}
Moreover,
\begin{equation}\label{eq:str.cv.sT}
  \forall \varphi\in W^{1,1}(\Omega)\cap W^{k+2,\infty}(\Th)\,:\,
  \sum_{T\in\Th}s_{T}(\IT\varphi,\IT\varphi)\to 0 \mbox{ as $h\to 0$}.
\end{equation}
\end{proposition}

\begin{proof} 
We write $A\lesssim B$ for $A\le MB$ where $M$ does not depend on $h$
or $\varphi$.

\emph{Step 1: Proof of~\eqref{eq:interp.error.GT}.}
By the commuting property~\eqref{eq:commut.GT} and the
approximation property \eqref{eq:approx.lproj.Wsp} applied to $v=\partial_i \varphi$, $s=k+1$ and $m=0$,
we have $\norm[L^p(T)^d]{\GT\IT\varphi-\GRAD\varphi}\lesssim h_T^{k+1}\norm[W^{k+2,p}(T)]{\varphi}$ for all $T\in\Th$.
Raising this inequality to the power $p$ and summing over $T\in\Th$
(if $p$ is finite, otherwise taking the maximum over $T\in\Th$) gives \eqref{eq:interp.error.GT}.

\medskip

\emph{Step 2: Proof of~\eqref{eq:str.cv.GT}.}
We reason by density. We take $(\varphi_\epsilon)_{\epsilon>0}\subset W^{k+2,p}(\Omega)$
that converges to $\varphi$ in $W^{1,p}(\Omega)$ as $\epsilon\to 0$ and
we write, inserting $\pm(\Gh\Ih\varphi_\epsilon-\GRAD\varphi_\epsilon)$ and using the triangle inequality,
\begin{align*}
\norm[L^p(\Omega)^d]{\Gh\Ih\varphi-\GRAD\varphi}
&\le \norm[L^p(\Omega)^d]{\Gh\Ih(\varphi-\varphi_\epsilon)}
+ \norm[L^p(\Omega)^d]{\Gh\Ih\varphi_\epsilon-\GRAD\varphi_\epsilon}
+ \norm[L^p(\Omega)^d]{\GRAD(\varphi_\varepsilon-\varphi)}\\
&\lesssim \norm[L^p(\Omega)^d]{\GRAD(\varphi-\varphi_\epsilon)}
+ \norm[L^p(\Omega)^d]{\Gh\Ih\varphi_\epsilon-\GRAD\varphi_\epsilon},
\end{align*}
where we have used the commuting property~\eqref{eq:commut.GT} followed by the $L^p$-stability of the $L^2$-projector stated in Lemma \ref{lem:stab.proj} to pass to the second line.
By \eqref{eq:interp.error.GT}, the second term in this right-hand side tends
to $0$ as $h\to 0$. Taking (in that order) the supremum limit as $h\to 0$
and then the supremum limit as $\epsilon\to 0$ concludes the proof
that $\Gh\Ih\varphi\to\GRAD\varphi$ in $L^p(\Omega)^d$.

\medskip

\emph{Step 3: Proof of~\eqref{eq:str.cv.sT}.}
It is proved in \cite[Eq. (46)]{Di-Pietro.Ern.ea:14} that
\[
h_F^{-\frac{1}{2}}\norm[L^2(F)]{\lproj[F]{k}((\IT\varphi)_F-\PT\IT\varphi)}
\lesssim h_T^{k+1}\norm[H^{k+2}(T)]{\varphi}.
\]
Using Lemma \ref{lem:lebemb}, the admissibility of the mesh (which gives
$h_F|F|\approx |T|$ if $F\in\Fh[T]$), and the regularity assumption on $\varphi$, we infer
\begin{align*}
  h_F^{1-p}\norm[L^p(F)]{\lproj[F]{k}((\IT\varphi)_F-\PT\IT\varphi)}^p
  &\lesssim h_F^{1-\frac{p}{2}} |F|^{1-\frac{p}{2}}
  \left(
  h_F^{-\frac{1}{2}}\norm[L^2(F)]{\lproj[F]{k}((\IT\varphi)_F-\PT\IT\varphi)}
  \right)^p\\
  &\lesssim (h_F|F|)^{1-\frac{p}{2}} h_T^{(k+1)p}
  \norm[H^{k+2}(T)]{\varphi}^p\\
  &\lesssim |T|^{1-\frac{p}{2}} h_T^{(k+1)p}|T|^{\frac{p}{2}}\norm[W^{k+2,\infty}(T)]{\varphi}^p\\
  &\lesssim |T| h^{(k+1)p}\norm[W^{k+2,\infty}(\Th)]{\varphi}^p.
\end{align*}
Summing this inequality over $F\in\Fh[T]$ and $T\in\Th$, and recalling the uniform bound~\eqref{eq:Np} over $\card{\Fh[T]}$, we get
$$
\sum_{T\in\Th}s_T(\IT\varphi,\IT\varphi)\lesssim
|\Omega|h^{(k+1)p}\norm[W^{k+2,\infty}(\Th)]{\varphi}^p,
$$
and the proof is complete.
\end{proof}


\section{Convergence analysis}\label{sec:conv.anal}

The following proposition contains an \emph{a priori} estimate, uniform in $h$, on the solution to the discrete problem \eqref{def:hho.scheme}.

\begin{proposition}[A priori estimates]\label{prop:apriori.est}
  Under Assumption \ref{def:adm.Th}, if $\su[h]\in\UhD$ solves \eqref{def:hho.scheme}, then
	there exists $C$ only depending on $\Omega$, $\coera$, $\varrho$, $k$ and $p$ such that
\begin{equation}
  \label{eq:apriori-est}
  \norm[1,p,h]{\su[h]}\le C \norm[L^{p'}(\Omega)]{f}^{\frac{1}{p-1}}.
\end{equation}
\end{proposition}

\begin{proof}
	We write $A\lesssim B$ for $A\le MB$ with $M$ having the same dependencies as
$C$ in the proposition.
  Plugging $\sv[h]=\su[h]$ into \eqref{eq:hho-glob} and using the coercivity \eqref{hyp:ac} of $\bfa$ leads to
  \[
  \coera \sum_{T\in\Th}\norm[L^p(T)^d]{ \GT\su }^p + \sum_{T\in\Th}
  \sum_{F\in\Fh[T]}h_F^{1-p}\norm[L^p(F)]{\lproj[F]{k}(\unu[F]-\PT\su)}^p
  \le
  \norm[L^{p'}(\Omega)]{f}\norm[L^p(\Omega)]{\unu[h]}.
  \]
  Recalling the norm equivalence \eqref{equiv:norms}, and using the discrete Sobolev embeddings~\eqref{eq:discsob} with $q=p$ to estimate the second factor in the right-hand side, this gives
  \[
  \norm[1,p,h]{\su[h]}^p
  \lesssim
  \norm[L^{p'}(\Omega)]{f}\norm[L^p(\Omega)]{\unu[h]}
  \lesssim\norm[L^{p'}(\Omega)]{f}\norm[1,p,h]{\su[h]},
  \]
  which concludes the proof since, by assumption, $p>1$.
\end{proof}

We can now prove that the discrete problem \eqref{def:hho.scheme} has at least one solution.

\begin{proof}[Proof of Theorem \ref{th:exist.hho}]
We use \cite[Theorem 3.3]{Deimling:95} (see also \cite{Leray.Lions:65}): If
$(E,\langle\cdot,\cdot\rangle_{E},\norm[E]{{\cdot}})$ is an Euclidean space, and $\Phi:E\to E$ is continuous and satisfies $\frac{\langle \Phi(x),x\rangle_{E}}{\norm[E]{x}}\to +\infty$ as $\norm[E]{x}\to +\infty$, then $\Phi$ is onto.
We take $E=\UhD$, endowed with an arbitrary inner product, and define
$\Phi:\UhD\to \UhD$ by
\[
\forall \sv[h],\sw[h]\in \UhD,\qquad
\langle \Phi(\sv[h]),\sw[h]\rangle_{E}
=\asch(\sv[h],\sw[h]).
\]
Assumptions \eqref{hyp:acarat} and \eqref{hyp:ag} show that $\Phi$ is continuous, and the coercivity \eqref{hyp:ac} of $\bfa$ together with the norm equivalence \eqref{equiv:norms} show that
\[
\langle \Phi(\sv[h]),\sv[h]\rangle_{E} \ge C \norm[1,p,h]{\sv[h]}^p\ge
C_{\Th}\norm[E]{\sv[h]}^p,
\]
where $C_{\Th}>0$ may depend on $\Th$ but does not depend on $\sv[h]$
(we use the equivalence of all norms on the finite-dimensional space $\UhD$). Hence, $\Phi$ is onto.
Let now $\underline{\mathsf{y}}_h\in\UhD$ be such that
$$
\langle \underline{\mathsf{y}}_h, \sw[h]\rangle_{E} 
= \int_{\Omega} f(\vec{x})\unw[h](\vec{x})d\vec{x}
\qquad\forall \sw[h]\in \UhD,
$$
and take $\su[h]\in \UhD$ such that $\Phi(\su[h])=\underline{\mathsf{y}}_h$.
By definition of $\Phi$ and $\underline{\mathsf{y}}_h$, $\su[h]$ is a solution
to the discrete problem \eqref{def:hho.scheme}.
\end{proof}

Let us now turn to the proof of convergence. To improve the legibility of certain formulas, we often drop the variable $\vec{x}$ inside integrals.

\begin{proof}[Proof of Theorem \ref{th:conv.hho}]

\emph{Step 1: Existence of a limit.}
By Propositions \ref{prop:apriori.est} and \ref{prop:comp}, there exists
$u\in W^{1,p}_0(\Omega)$ such that up to a subsequence
as $h\to 0$, $\unu[h]\to u$ and $\ph\su[h]\to u$ in $L^q(\Omega)$ for all $q<p^*$,
and $\Gh\su[h]\to \GRAD u$ weakly in $L^p(\Omega)^d$.
Let us prove that $u$ solves \eqref{eq:weak}.
To this end, we adapt Minty's technique \cite{Minty:63,Leray.Lions:65}
to the discrete setting, as previously done in \cite{Droniou:06,Doniou.Eymard.et.al:12}.

\medskip

\emph{Step 2: Identification of the limit.}
The growth assumption \eqref{hyp:ag} on $\bfa$ ensures that
$\bfa(\cdot,\unu[h],\Gh\su[h])$ is bounded in $L^{p'}(\Omega)^d$,
and converges therefore (upon extracting another subsequence) to some $\bchi$ weakly in this space.
Let $\varphi\in C^\infty_c(\Omega)$.
Plugging $\sv[h]=\Ih\varphi$ into \eqref{def:hho.scheme} gives
\begin{equation}\label{eq:ident.lim.chi}
\int_\Omega \bfa(\vec{x},\unu[h],\Gh\unu[h])\SCAL
\Gh\Ih\varphi=\int_\Omega f \lproj{k}\varphi
-\sum_{T\in\Th}s_T(\su[T],\IT\varphi),
\end{equation}
with $\lproj{k}$ denoting the $L^2$-projector on the broken polynomial space $\Poly{k}(\Th)$.
Using H\"older's inequality followed by the norm equivalence~\eqref{equiv:norms} to bound the first factor, we infer
$$
\begin{aligned}
  \left|\sum_{T\in\Th}s_T(\su,\IT\varphi)\right|
  &\le
  \left(\sum_{T\in\Th}s_T(\su,\su)\right)^{\frac{1}{p'}}
  \left(\sum_{T\in\Th}s_T(\IT\varphi,\IT\varphi)\right)^{\frac{1}{p}} 
  \\
  &\le \norm[1,p,h]{\su[h]}^{\frac{p}{p'}}
  \left(\sum_{T\in\Th}s_T(\IT\varphi,\IT\varphi)\right)^{\frac{1}{p}}.
\end{aligned}
$$
Recalling the a priori bound \eqref{eq:apriori-est} on the exact solution and the strong convergence property~\eqref{eq:str.cv.sT}, we see that this quantity tends to $0$ as $h\to 0$.
Additionally, by the approximation properties of the $L^{2}$-projector stated in Lemma \ref{lem:Wkp.interp} together with the strong convergence property~\eqref{eq:str.cv.GT}, we have $\lproj{k}\varphi\to \varphi$ in $L^p(\Omega)$
and $\Gh\Ih\varphi\to \GRAD\varphi$ in $L^p(\Omega)^d$. We can therefore pass to the limit
$h\to 0$ in \eqref{eq:ident.lim.chi}, and we find
\begin{equation}\label{eq:ident.lim.chi2}
\int_\Omega \bchi\SCAL \GRAD\varphi=\int_\Omega f\varphi.
\end{equation}
By density of $C^\infty_c(\Omega)$ in $W^{1,p}_0(\Omega)$, this relation still holds if $\varphi\in W^{1,p}_0(\Omega)$.

Let us now take $\bL\in L^p(\Omega)^d$ and write, using the monotonicity \eqref{hyp:am} of $\bfa$,
\begin{equation}\label{minty:1}
\int_\Omega [\bfa(\vec{x},\unu[h],\Gh\su[h])-\bfa(\vec{x},\unu[h],\bL)]\cdot[\Gh\su[h]-\bL]\ge 0.
\end{equation}
Use \eqref{def:hho.scheme} and $s_T(\su,\su)\ge 0$
to write
\begin{equation}
\int_\Omega \bfa(\vec{x},\unu[h],\Gh\su[h])\cdot\Gh\su[h]
=\int_\Omega f\unu[h] -\sum_{T\in\Th}s_T(\su,\su)
\le \int_\Omega f\unu[h].
\label{minty.10}\end{equation}
Develop \eqref{minty:1} and plug this relation:
\begin{equation}\label{minty:2}
\int_\Omega f\unu[h]-\int_\Omega \bfa(\vec{x},\unu[h],\Gh\su[h])\SCAL\bL
\ge \int_\Omega \bfa(\vec{x},\unu[h],\bL)\cdot[\Gh\su[h]-\bL].
\end{equation}
Since $\unu[h]\to u$ in $L^q(\Omega)$ for all $q<p^*$, the Caratheodory and growth properties \eqref{hyp:acarat}
and \eqref{hyp:ag} of $\bfa$ show that $\bfa(\vec{x},\unu[h],\bL)\to \bfa(\vec{x},u,\bL)$ strongly in $L^{p'}(\Omega)^d$.
We can therefore pass to the limit in \eqref{minty:2}:
\begin{equation}\label{u.sol}
\int_\Omega f u-\int_\Omega \bchi\cdot\bL\ge \int_\Omega \bfa(\vec{x},u,\bL)\cdot[\GRAD u-\bL].
\end{equation}
The conclusion then follows classically  \cite{Minty:63,Leray.Lions:65}:
Take $v\in W^{1,p}_0(\Omega)$, apply this relation to $\bL=\GRAD u \pm t\GRAD v$
for some $t>0$, use \eqref{eq:ident.lim.chi2} with $\varphi=u\pm t v$, divide by $t$, and let $t\to 0$
using the Caratheodory and growth properties of $\bfa$.
This leads to
\[
\int_\Omega f v = \int_\Omega \bfa(\vec{x},u,\GRAD u)\cdot \GRAD v,
\]
and the proof that $u$ solves \eqref{eq:weak} is complete.

\medskip

\emph{Step 3: Convergence of the gradient.}
It remains to show that if $\bfa$ is strictly monotone, then $\Gh\su[h]
\to \GRAD u$ strongly in $L^p(\Omega)^d$. Let
\begin{equation}\label{str.cv.grad.1}
F_h=[\bfa(\vec{x},\unu[h],\Gh\su[h])-\bfa(\vec{x},\unu[h],\GRAD u)]\cdot[\Gh\su[h]-\GRAD u]\ge 0
\end{equation}
Developing this expression and using \eqref{minty.10}, we can pass to the limit and use
\eqref{eq:ident.lim.chi2} to see that
\[
\limsup_{h\to 0}\int_\Omega F_h\le \int_\Omega fu-\int_\Omega \bchi\SCAL\GRAD u=0.
\]
Hence, $F_h\to 0$ in $L^1(\Omega)$. Up to a subsequence, it therefore converges almost everywhere.
Using the coercivity and growth assumptions \eqref{hyp:ac} and \eqref{hyp:ag} of $\bfa$,
Young's inequality gives
\begin{align}
F_h&\ge \coera |\Gh\su[h]|^p-(\gra(\vec{x})+ \upa|\unu[h]|^r+\upa|\Gh\su[h]|^{p-1})|\GRAD u|
-(\gra(\vec{x})+ \upa|\unu[h]|^r+\upa|\GRAD u|^{p-1})|\GRAD u|\nonumber\\
&\ge \frac{\coera}{2} |\Gh\su[h]|^p-2(\gra(\vec{x})+ \upa|\unu[h]|^r)|\GRAD u|-\upa|\GRAD u|^p
-\frac{\upa^p}{p}\left(\frac{2}{p'\coera}\right)^{p-1}|\GRAD u|^p.\label{str.cv.grad.2}
\end{align}
Since, up to a subsequence, $\unu[h]$ converges a.e., this relation shows that
for a.e. $\vec{x}$, the sequence $(\Gh\su[h](\vec{x}))_{h\in\mathcal H}$ remains bounded.
Let us show that it can only have $\GRAD u(\vec{x})$ as adherence value. If $\zeta$ is
an adherence value of $(\Gh\su[h](\vec{x}))_{h\in\mathcal H}$, then, passing to the limit
in \eqref{str.cv.grad.1} gives, since $F_h\to 0$ and $\unu[h]\to u$ a.e., 
\[
[\bfa(\vec{x},u(\vec{x}),\zeta)-\bfa(\vec{x},u(\vec{x}),\GRAD u(\vec{x}))]\cdot[\zeta-\GRAD u(\vec{x})]=0.
\]
The strict monotonicity of $\bfa$ then shows that $\zeta=\GRAD u(\vec{x})$. Hence,
for a.e. $\vec{x}$, the bounded sequence $(\Gh\su[h](\vec{x}))_{h\in\mathcal H}$ has only $\GRAD u(\vec{x})$ as adherence value, and thus $\Gh\su[h]\to \GRAD u$ a.e. on $\Omega$.

Since $(F_h)_{h\in\mathcal H}$ is 1-equi-integrable (it converges in $L^1(\Omega)$) and
$(|\unu[h]|^r)_{h\in\mathcal H}$ is $p'$-equi-integrable ($p'r<p^*$ and $(\unu[h])_{h\in\mathcal H}$
therefore converges in $L^{p'r}(\Omega)$), \eqref{str.cv.grad.2} shows that $(\Gh\su[h])_{h\in\mathcal H}$
is $p$-equi-integrable. Vitali's theorem then gives the strong convergence of this sequence to $\GRAD u$ in $L^p(\Omega)^d$. \end{proof}


\section{Other boundary conditions}\label{sec:other.bcs}

We briefly discuss here how the HHO scheme is written for non-homogeneous Dirichlet and homogeneous Neumann boundary conditions and hint at the modifications required in the convergence proof.

\subsection{Non-homogeneous Dirichlet boundary conditions}

Non-homogeneous Dirichlet boundary conditions consist in replacing
\eqref{eq:strong:BC} with
\begin{equation}\label{eq:strong:nhd}
  u=g\mbox{ on $\partial\Omega$}
\end{equation}
with $g\in W^{1-\frac{1}{p},p}(\partial\Omega)$.
Denoting by $\tr:W^{1,p}(\Omega)\to W^{1-\frac{1}{p},p}(\partial\Omega)$ the
trace operator, the weak formulation becomes:
\begin{equation}
  \label{eq:weak:nhd}
  \begin{array}{l}
    \mbox{Find }u\in W^{1,p}(\Omega)\mbox{ such that }\tr(u)=g\mbox{ and, for all $v\in W^{1,p}_0(\Omega)$},\\
    \displaystyle \int_\Omega \bfa(\vec{x},u(\vec{x}),\GRAD u(\vec{x}))\cdot
    \GRAD v(\vec{x})\dx=\int_\Omega f(\vec{x})v(\vec{x})\dx.
  \end{array}
\end{equation}
As in Remark \ref{rem:domain.interp} we notice that $\lproj[F]{k}g$ is well defined for any $F\in\Fhb$.
Hence, we can define the vector $\su[g,h]\in\Uh$ such that
\[
\unu[g,T]=0\quad\forall T\in\Th,\qquad
\unu[g,F]=0\quad\forall F\in\Fhi,\qquad
\unu[g,F]=\lproj[F]{k}g\quad\forall F\in\Fhb.
\]
We then set
\[
\UhD[g]\eqbydef\UhD+\su[g,h],
\]
and write the discrete problem corresponding to \eqref{eq:weak:nhd} as
\begin{equation}
  \label{eq:hho-glob:nhd}
  \mbox{Find $\su[h]\in \UhD[g]$ such that, for any $\sv[h]\in \UhD$, }\asch(\su[h],\sv[h])
  = \int_{\Omega} f \unv[h],
\end{equation}
with $\asch$ defined by \eqref{eq:hho-assembly}--\eqref{eq:hho-loc}.
The convergence analysis for non-homogeneous Dirichlet boundary conditions
is performed as usual by utilizing a lifting of the boundary conditions.
We take $\widetilde{g}\in W^{1,p}(\Omega)$ and let $\sg[h]=\Ih\widetilde{g}$.
Making $\sv[h]=\su[h]-\sg[h]\in\Uh$ in \eqref{eq:hho-glob:nhd} and
using $\norm[1,p,h]{\sg[h]}\lesssim \norm[W^{1,p}(\Omega)]{g}$ (see
Proposition \ref{prop:stab.interp} below) enables us to prove 
\emph{a priori} estimates on $\norm[1,p,h]{\su[h]-\sg[h]}$.

Proposition \ref{prop:str.cv.interp} does not rely on the homogeneous
boundary conditions and therefore shows that $\GT\sg[h]\to \GRAD \widetilde{g}$ in $L^p(\Omega)^d$
as $h\to 0$. Since $\lproj[h]{k}\widetilde{g}\to \widetilde{g}$ in $L^p(\Omega)$
(see Lemma \ref{lem:Wkp.interp}), applying Proposition \ref{prop:comp}
to $\sv[h]=\su[h]-\sg[h]$ shows that, for some $u\in W^{1,p}(\Omega)$ such
that $u-\widetilde{g}\in W^{1,p}_0(\Omega)$ (i.e. $\gamma(u)=g$),
up to a subsequence $\unu[h]\to u$ in $L^p(\Omega)$ and $\GT\su[h]\to \GRAD u$
in $L^p(\Omega)^d$ as $h\to 0$. The proof that $u$ is a solution
to \eqref{eq:weak:nhd} is then done in a similar way as for homogeneous
boundary conditions.

\begin{proposition}[Discrete norm estimate for interpolate of $W^{1,p}$ functions]
  \label{prop:stab.interp}
  Let $(\Th)_{h\in\mathcal H}$ be an admissible mesh sequence, and let $k\in\mathbb{N}$.
  Let $v\in W^{1,p}(\Omega)$ and let $\Ih v\in \Uh$ be the interpolant
  defined by \eqref{eq:Ih} and \eqref{eq:IT}. Then,
  $\norm[1,p,T]{\IT v}\lesssim \norm[W^{1,p}(T)]{v}$ for all $T\in\Th$, and thus
  $\norm[1,p,h]{\Ih v}\lesssim \norm[W^{1,p}(\Omega)]{v}$.
\end{proposition}

\begin{proof} 
  Set $\sv[h]\eqbydef\Ih v$ and let $T\in\Th$.
  Since $\unv[T]=\lproj[T]{k}v$, Corollary \ref{cor:Wsp.stab} with $s=1$ shows that
  $\norm[L^p(T)]{\GRAD \unv[T]}\lesssim \norm[W^{1,p}(T)]{v}$. This takes care of the first
  term in $\norm[1,p,T]{\sv}$. To deal with the second term, we use 
  Lemma \ref{lem:stab.proj} with $U=F$ and then Lemma 
  \ref{lem:Wkp.interp.trace} with $m=0$ and $s=1$ to write
  \[
  \norm[L^p(F)]{\unv[F]-\unv[T]}=\norm[L^p(F)]{\lproj[F]{k}v-\lproj[T]{k}v}
  =\norm[L^p(F)]{\lproj[F]{k}(v-\lproj[T]{k}v)}
  \lesssim \norm[L^p(F)]{v-\lproj[T]{k}v}
  \lesssim h_T^{1-\frac1p}\norm[W^{1,p}(T)]{v}.
  \]
  Raising this to the power $p$ and using $h_T\lesssim h_F$ gives
  $h_F^{1-p}\norm[L^p(F)]{\unv[F]-\unv[T]}^p\lesssim \norm[W^{1,p}(T)]{v}^p$.
  The global bound is then inferred raising the local bounds to the power $p$ and summing over $T\in\Th$.
\end{proof}

\subsection{Homogeneous Neumann boundary conditions}

We assume that
\[
\int_\Omega f(\vec{x})\dx=0.
\]
Homogeneous Neumann boundary conditions for elliptic Leray--Lions problems consist in replacing
\eqref{eq:strong:BC} with
\begin{equation}\label{eq:strong:neu}
  \bfa(\cdot,u,\GRAD u)\SCAL\normal=0\mbox{ on $\partial\Omega$},
\end{equation}
where $\normal$ is the outer normal to $\partial\Omega$. The weak
formulation of \eqref{eq:strong:PDE}--\eqref{eq:strong:neu} is
\begin{equation}
  \label{eq:weak:neu}
  \begin{array}{l}
    \text{Find $u\in W^{1,p}(\Omega)$ such that $\int_\Omega u(\vec{x})\dx=0$ and, for all $v\in W^{1,p}(\Omega)$},\\
    \displaystyle \int_\Omega \bfa(\vec{x},u(\vec{x}),\GRAD u(\vec{x}))\cdot
    \GRAD v(\vec{x})\dx=\int_\Omega f(\vec{x})v(\vec{x})\dx.
  \end{array}
\end{equation}
The HHO scheme for \eqref{eq:weak:neu} reads
\begin{equation}
  \label{eq:hho-glob:nhn}
  \begin{array}{l}
    \displaystyle  \mbox{Find $\su[h]\in \Uh$ such that }\int_\Omega \unu[h](\vec{x})\dx=0\mbox{ and, for any $\sv[h]\in \Uh$, }
    \displaystyle\asch(\su[h],\sv[h])  = \int_{\Omega} f \unv[h]
  \end{array}
\end{equation}
with $\asch$ still defined by \eqref{eq:hho-assembly}--\eqref{eq:hho-loc}.

To carry out the convergence analysis from Section \ref{sec:conv.anal}, we need a few
results. The first one is a discrete Poincar\'e--Wirtinger--Sobolev inequality, which bounds
to the $L^{p^*}$-norm of discrete functions by their discrete norm. This immediately
gives \emph{a priori} estimates on the solution to the scheme (Proposition \ref{prop:apriori.est}).
The second result is a discrete Rellich theorem for functions with zero average and
bounded discrete norm (this is the equivalent of Proposition \ref{prop:comp}).
The proofs of both results are based on Lemma \ref{lem:poi-wir} and on a decomposition
of functions in $\Uh$ into low-order (piecewise-constant) vectors in $\Uhzero$,
and their higher order variation.

\begin{lemma}[Discrete Poincar\'e--Wirtinger--Sobolev inequality for broken polynomial
    functions with zero global average]\label{lem:poi-wir-glob}
  Let $(\Th)_{h\in\mathcal{H}}$ be an admissible mesh sequence, and let $q=p^*$
if $p\neq d$, and $q\in [1,+\infty)$ if $p=d$.
  Then, there exists $C$ only depending on $\Omega$, $\varrho$, $k$, $q$ and $p$ such that,
  for all $\sv[h]\in\Uh$ satisfying $\int_\Omega \unv[h](\vec{x})\dx=0$, we have
  \begin{equation}\label{eq:poi-wir-glob}
    \norm[L^{q}(\Omega)]{\unv[h]}
    \le C \norm[1,h,p]{\sv[h]}.
  \end{equation}
\end{lemma}

\begin{proof} Here, $A\lesssim B$ means that $A\le MB$ with $M$ only depending on
  $\Omega$, $\varrho$, $k$ and $p$.
  We define $\sv[h]^0\in\Uhzero$ and $\unv[h]^1\in \Poly{k}(\Th)$ by:
  \begin{gather*}
    \unv[T]^0=\lproj[T]{0}\unv[T]\quad\forall T\in\Th\,,\qquad
    \unv[F]^0=\lproj[F]{0}\unv[F]\quad\forall F\in\Fh,
    \\
    \unv[T]^1=\unv[T]-\lproj[T]{0}\unv[T]=\unv[T]-\unv[T]^0\quad\forall T\in\Th.
  \end{gather*}
  By Lemma \ref{lem:poi-wir} we have
  \begin{equation}\label{PSW.1}
    \norm[L^{q}(\Omega)]{\unv[h]^1}\lesssim \left(\sum_{T\in\Th}\norm[L^p(T)]{\GRAD \unv[T]}^p\right)^{1/p}.
  \end{equation}
  We recall the definition of the discrete $W^{1,p}$-norm on $\Uhzero$ from~\cite{koala}:
  \[
  \norm[W^{1,p},\Th]{\sv[h]^0}=\left(\sum_{T\in\Th}\sum_{F\in\Fh[T]}|T|\left|\frac{\unv[T]^0-\unv[F]^0}{h_T}\right|^p\right)^{\frac1p}
  \]
  (the genuine discrete $W^{1,p}$-norm in \cite{koala} involves a different coefficient
  than $|T|$ in this sum, but under Assumption \ref{def:adm.Th} this coefficient
  is $\approx |T|$).
  Since $\sum_{T\in\Th}|T|\unv[T]^0=\int_\Omega  \unv[h](\vec{x})d\vec{x}=0$, \cite{koala} gives
  \begin{equation}\label{PSW.2}
    \norm[L^{q}(\Omega)]{\unv[h]^0}\lesssim \norm[W^{1,p},\Th]{\sv[h]^0}.
  \end{equation}
  By noticing that $\unv[h]=\unv[h]^0+\unv[h]^1$, the result follows from \eqref{PSW.1} and \eqref{PSW.2} provided that
  \begin{equation}\label{PSW.3}
    \norm[W^{1,p},\Th]{\sv[h]^0}\lesssim \norm[1,h,p]{\sv[h]}.
  \end{equation}
  An easy generalisation of \cite[Lemma 6.3]{Droniou.Eymard:06} and
  \cite[Lemma 6.6]{Droniou.Eymard:09} (see \cite{koala} for details) shows that
  \[
  \left|\lproj[F]{0}\unv[T]-\lproj[T]{0}\unv[T]\right|^p
  =\left|\frac{1}{|F|}\int_F \unv[T](\vec{x})\dsx-
  \frac{1}{|T|}\int_T\unv[T](\vec{x})d\vec{x}\right|^p
  \lesssim \frac{h_T^p}{|T|}\int_T |\GRAD \unv[T](\vec{x})|^pd\vec{x}.
  \]
  Using the triangular and Jensen's inequalities, and the relations $|T|\lesssim |F|h_F$
  and $h_F\le h_T$, we infer
  \begin{align*}
    |\unv[F]^0-\unv[T]^0|^p&\lesssim \left|\lproj[F]{0} \unv[F]-\lproj[F]{0}\unv[T]\right|^p +
    \frac{h_T^p}{|T|}\int_T |\GRAD \unv[T](\vec{x})|^pd\vec{x}\\
    &\lesssim \frac{1}{|F|}\int_F |\unv[F](\vec{x})-\unv[T](\vec{x})|^p\dsx
    +\frac{h_T^p}{|T|}\norm[L^p(T)^d]{\GRAD \unv[T]}^p\\
    &\lesssim \frac{h_T^p}{|T|}h_F^{1-p}\norm[L^p(F)]{\unv[F]-\unv[T]}^p
    +\frac{h_T^p}{|T|}\norm[L^p(T)^d]{\GRAD \unv[T]}^p.
  \end{align*}
  Multiplying by $\frac{|T|}{h_T^p}$ and summing over $F\in\Fh[T]$ and $T\in\Th$
  gives \eqref{PSW.3}.
\end{proof}

\begin{proposition}[Compactness result for broken polynomial
    function with zero global average]
  Let $(\Th)_{h\in\mathcal{H}}$ be an admissible mesh sequence and let
  $\sv\in \Uh$ be such that $(\norm[1,h,p]{\sv[h]})_{h\in\mathcal H}$ is bounded
  and, for all $h\in\mathcal H$, $\int_\Omega \unv[h](\vec{x})\dx=0$.
  Then, there exists $v\in W^{1,p}(\Omega)$ such that $\int_\Omega v(\vec{x})d\vec{x}=0$ and, up to a subsequence as $h\to 0$, recalling the definition~\eqref{eq:p*} of the Sobolev index $p^*$,
  \begin{itemize}
  \item $\unv[h]\to v$ and $\ph\sv[h]\to v$ strongly in $L^q(\Omega)$ for all $q<p^*$,
  \item $\Gh\sv[h]\to \GRAD v$ weakly in $L^p(\Omega)^d$.
  \end{itemize}
\end{proposition}

\begin{proof}
  We use the same decomposition $\unv[h]=\unv[h]^0+\unv[h]^1$ as in the proof of Lemma
  \ref{lem:poi-wir-glob}. By Lemma \ref{lem:poi-wir} we have
  $\norm[L^q(\Omega)]{\unv[h]^1}\le C h^\theta \norm[1,h,p]{\sv[h]}$ where $C$ does
  not depend on $h$ and $\theta=1+\frac{d}{q}-\frac{d}{p}>0$. Hence, $\unv[h]^1
  \to 0$ in $L^q(\Omega)$ as $h\to 0$. By \eqref{PSW.3}, $(\norm[W^{1,p},\Th]{\sv[h]^0})_{h\in\mathcal H}$
  remains bounded. Since $\sum_{T\in\Th} |T|\unv[h]^0=0$ for all $h\in\mathcal H$,
  the discrete compactness result for Neumann boundary conditions of \cite{koala}
  shows that there exists a $v\in W^{1,p}(\Omega)$ with zero average such that
  $\unv[h]^0\to v$ strongly in $L^q(\Omega)$ up to a subsequence. Hence, $\unv[h]\to v$
  in $L^q(\Omega)$ along the same subsequence. We then apply Corollary \ref{cor:comp-u-pu},
  which is independent of the boundary conditions, to deduce that $\ph\sv[h]\to v$
  in $L^q(\Omega)$. 

  To prove that $\GT[h]\sv[h]\to \GRAD v$ weakly in $L^p(\Omega)^d$, we notice that
  by Lemma \ref{lem:equivnorms} the functions $\GT[h]\sv[h]$ remain bounded in $L^p(\Omega)^d$
  and therefore converge weakly to some $\mathcal G$ in this space. We prove that $\mathcal G=
  \GRAD v$ as in the proof of Proposition \ref{prop:comp}, using test functions $\bphi\in C^\infty_c(\Omega)^d$ instead of $\bphi\in C^\infty(\Real^d)^d$. \end{proof}

\section{Conclusion}\label{sec:ccl}

We extended the HHO method of \cite{Di-Pietro.Ern.ea:14} to fully non-linear Leray--Lions equations,
which include the $p$-Laplace model. The lowest-order version of this method (corresponding to $k=0$) belongs to the family of mixed-hybrid Mimetic Finite Differences, Hybrid Finite Volumes and Mixed Finite Volumes
schemes. We proved the convergence of the HHO method without assuming unrealistic
regularity properties on the solution, or restrictive assumptions on the non-linear operator.
To establish this convergence, we developed discrete functional analysis results that include
the analysis of $L^p$- and $W^{s,p}$-stability and approximation properties of $L^2$-projectors on broken polynomial spaces. We provided numerical results which demonstrate the good approximation properties of the method on a variety of meshes, and for various orders
(low as well as high).


\appendix

\section{Discrete functional analysis in local polynomial spaces}\label{sec:appen}

This appendix collects discrete functional analysis results in local polynomial spaces that are of general interest for polynomial-based discretizations of linear and nonlinear problems.
Most of these results have already been stated without proof in the paper, but we restate them for the sake of easy consultation.

\subsection{Estimates in local polynomial spaces}\label{sec:est.loc.space}

This section collects $L^p$- and $W^{s,p}$-estimates in local polynomial spaces including direct and reverse Sobolev and Lebesgue embeddings.

\lebemb*

\begin{remark}[Reverse embeddings]
  If $q\le m$ then this result is a classical (direct) Lebesgue
  embedding due to H\"older's inequality. It holds for $m<q$ solely because we consider
  polynomials (and we notice that the scaling $|U|^{\frac{1}{q}-\frac{1}{m}}$ explodes
  as $h_U\to 0$).
\end{remark}

\begin{remark}[Sobolev reverse embeddings]
  Let $U$ be a polyhedral set that admits a simplicial decomposition 
  such that for any simplex $S$, if $h_S$ is the diameter of $S$ and $r_S$
  its inradius then $h_S\le \varrho r_S$, and $h_U\le \varrho h_S$.
  The following inverse inequality holds with $C_{\rm inv}$ depending on
  $\varrho$, $k$ and $p$, but independent of $h$
  (cf.~\cite[Lemma~1.44]{Di-Pietro.Ern:12} for the case $p=2$ and use
  use \cite[Lemma~1.50]{Di-Pietro.Ern:12} or Lemma \ref{lem:lebemb} to deduce the general case), 
  \begin{equation}
    \label{eq:inv}
    \forall v\in\Poly{k}(U)\,:\,
    \norm[L^p(U)]{\GRAD v}\le 
    C_{\rm inv} h_U^{-1}\norm[L^p(U)]{v}.
  \end{equation}
  Using this inequality, we can easily deduce from Lemma
  \ref{lem:lebemb} the following reverse Sobolev embeddings: Under the assumptions
  of Lemma \ref{lem:lebemb}, if $U$ is open and $m\ge r$, then for all $w\in \Poly[N]{k}(U)$ we have
  \[
  \seminorm[W^{m,p}(U)]{w}\lesssim h_U^{r-m} |U|^{\frac{1}{p}-\frac{1}{q}}\seminorm[W^{r,q}(U)]{w}.
  \]
  Here $\lesssim$ is up to a multiplicative constant only depending on $k$, $\delta$, 
  $p$, $q$ and $r$.
  Note that the result obviously cannot hold if $m<r$ and $m\le k$ (consider $w$ polynomial
  of degree exactly $m$: the left-hand side does not vanish, while the right-hand side does).
\end{remark}

\begin{proof}[Proof of Lemma~\ref{lem:lebemb}] 
  We obviously only have to prove $\lesssim$ since $m$ and $q$ play
  symmetrical roles in \eqref{rev:leb}.
  By~\eqref{eq:reg.U}, there is $\vec{x}_U\in U$ such that $B(\vec{x}_U,\delta h_U)\subset U\subset B(\vec{x}_U,h_U)$.
  Let $U_0=(U-\vec{x}_U)/h_U$.
  Using the change of variable $\vec{x}\in U\mapsto (\vec{x}-\vec{x}_U)/h_U\in U_0$, we see that, for $\ell\in [1,+\infty]$,
  \begin{equation}\label{leb:scaling}
    \norm[L^\ell(U)]{w}
    =h_U^\frac{N}{\ell}\norm[L^\ell(U_0)]{w_0}\approx |U|^{\frac{1}{\ell}}
    \norm[L^\ell(U_0)]{w_0},
  \end{equation}
  where we used $h_U^N\approx |U|$ (since $h_U\approx r_U$) and we set $w_0(\vec{y})=
  w(\vec{x}_U+h_U\vec{y})$. 
  Assume that there exists $C_0$ not depending on the geometry of $U_0$ but solely on $\delta$ such that
  \begin{equation}\label{leb:unit}
    \forall v\in\Poly[N]{k}(U_0)\,:\,
    \norm[L^q(U_0)]{v}\le C_0\norm[L^m(U_0)]{v}.
  \end{equation}
  Then combining this with \eqref{leb:scaling}, since $w_0\in \Poly[N]{k}(U_0)$,
  \[
  \norm[L^q(U)]{w}
  \lesssim |U|^{\frac{1}{q}}\norm[L^q(U_0)]{w_0}
  \lesssim |U|^{\frac{1}{q}}\norm[L^m(U_0)]{w_0}
  \lesssim |U|^{\frac{1}{q}-\frac{1}{m}}\norm[L^m(U)]{w},
  \]
  and the lemma is proved.

  It remains to establish \eqref{leb:unit}. To this end, we notice that,
  by choice of $\vec{x}_U$, we have $B(0,\delta)\subset U_0\subset B(0,1)$.
  Since $\norm[L^q(B(0,1))]{{\cdot}}$ and $\norm[L^m(B(0,\delta))]{{\cdot}}$ are
  both norms on $\Poly[N]{k}(U_0)$ (any polynomial that vanishes on a ball
  vanishes everywhere), and since $\Poly[N]{k}(U_0)$ is a finite-dimensional vector
  space, we have
  \begin{equation}\label{leb:unit2}
    \forall v\in \Poly[N]{k}(U_0)\qquad
    \norm[L^q(B(0,1))]{v}\lesssim \norm[L^m(B(0,\delta))]{v},
  \end{equation}
  with constant in $\lesssim$ depending on $\delta$ but not on the geometry of $U_0$. To prove~\eqref{leb:unit}, write
  \[
  \norm[L^q(U_0)]{v}
  \le\norm[L^q(B(0,1))]{v}
  \lesssim\norm[L^m(B(0,\delta))]{v}
  \le\norm[L^m(U_0)]{v}.\qedhere
  \]
\end{proof}

\subsection{$L^p$-stability and $W^{s,p}$-approximation properties of $L^2$-projectors} \label{sec:LPest.L2proj}

This section collects the proofs of $L^p$- and $W^{s,p}$-stability and approximation estimates for $L^2$-projectors on local polynomial spaces stated in Section~\ref{sec:local.poly}. 

\stabproj*

\begin{proof}

  In this proof, $A\lesssim B$ means that $A\le MB$ for some $M$ only depending on
  $N$, $\delta$, $k$ and $p$.

  \emph{Step 0: $p=2$.} This case is trivial since $\lproj[U]{k}$ is an
  orthogonal projector in $L^2(U)$ and therefore satisfies \eqref{stab.proj}
  with $C=1$.

  \emph{Step 1: $p>2$.} We use Lemma \ref{lem:lebemb} to write $\norm[L^p(T)]{\lproj[U]{k}g}
  \lesssim |T|^{\frac{1}{p}-\frac{1}{2}}\norm[L^2(T)]{\lproj[U]{k}g}$.
  Since $g\in L^p(T)\subset L^2(T)$, we can use \eqref{stab.proj} for $p=2$ and
  we deduce
  $\norm[L^p(T)]{\lproj[U]{k}g}\lesssim |T|^{\frac{1}{p}-\frac{1}{2}}\norm[L^2(T)]{g}$.
  We then conclude thanks to H\"older's inequality, valid since $p>2$,
  \[
  \norm[L^p(T)]{\lproj[U]{k}g}\lesssim |T|^{\frac{1}{p}-\frac{1}{2}}
  |T|^{\frac{1}{2}-\frac{1}{p}}\norm[L^p(T)]{g}=\norm[L^p(T)]{g}.
  \]

  \emph{Step 2: $p<2$.} We use a standard duality technique. Let $g\in L^p(U)$ and $w\in L^{p'}(U)$.
  Then by definition of $\lproj[U]{k}$ and using \eqref{stab.proj} with
  $p'>2$ instead of $p$,
  \[
  \int_U \lproj[U]{k}g(\vec{x})w(\vec{x})\,d\vec{x}
  =\int_U g(\vec{x})\lproj[U]{k}w(\vec{x})\,d\vec{x}
  \le \norm[L^p(U)]{g}\norm[L^{p'}(U)]{\lproj[U]{k}w}
  \lesssim \norm[L^p(U)]{g}\norm[L^{p'}(U)]{w}.
  \]
  Taking the supremum of this inequality over all $w\in L^{p'}(U)$
  such that $\norm[L^{p'}(U)]{w}=1$ shows that \eqref{stab.proj} holds.\end{proof}

\Wkpinterp*

\begin{proof}

  Here, $A\lesssim B$ means that $A\le MB$ with $M$ only depending on $N$, $\rho$, $k$, $s$ and $p$.

  The proof combines averaged Taylor polynomials \cite{Brenner.Scott:08,Dupont.Scott:80}
  with the $L^p$-stability of the $L^2$-projector (Lemma \ref{lem:stab.proj}). 
  Since smooth functions are dense in $W^{s,p}(U)$, we only need to prove the
  result for $v\in C^\infty(U)\cap W^{s,p}(U)$. The Sobolev representation of $v$ reads
  \cite{Brenner.Scott:08}
  \begin{equation}\label{sob.rep}
    v=Q^s v + R^s v
  \end{equation}
  where $Q^sv$ is a polynomial of degree less than or equal to $s-1$ and the remainder $R^sv$
  satisfies \cite[Lemma 4.3.8]{Brenner.Scott:08}
  \begin{equation}\label{sob.rep.rem}
    \forall r\in \{0,\ldots,s\}\,:\,\seminorm[W^{r,p}(U)]{R^s v}\lesssim
    h_U^{s-r}\seminorm[W^{s,p}(U)]{v}. 
  \end{equation}
  Since $Q^sv$ is a polynomial of degree $\leq s-1\leq k$, $\lproj[U]{k}(Q^sv)=Q^sv$ and therefore,
  from \eqref{sob.rep}, $\lproj[U]{k}v=Q^sv + \lproj[U]{k}(R^sv)$.
  Subtracting this from \eqref{sob.rep}, we infer $v-\lproj[U]{k}v=R^sv-\lproj[U]{k}(R^sv)$.
  Hence,
  \begin{equation}\label{interp.Wsp.1}
    \seminorm[W^{m,p}(U)]{v-\lproj[U]{k}v}\le
    \seminorm[W^{m,p}(U)]{R^sv}+ \seminorm[W^{m,p}(U)]{\lproj[U]{k}(R^sv)}.
  \end{equation}
  Iterating the inverse inequality \eqref{eq:inv} and using Lemma \ref{lem:stab.proj} we see that
  \begin{equation}\label{interp.Wsp.2}
    \seminorm[W^{m,p}(U)]{\lproj[U]{k}(R^sv)}\lesssim h_U^{-m}\norm[L^p(U)]{\lproj[U]{k}(R^sv)}
    \lesssim h_U^{-m}\norm[L^p(U)]{R^sv}.
  \end{equation}
  Estimate \eqref{sob.rep.rem} applied to $r=m$ and $r=0$ shows that
  \begin{equation}\label{interp.Wsp.3}
    \seminorm[W^{m,p}(U)]{R^sv}+h_U^{-m}\norm[L^p(U)]{R^sv}
    \lesssim h_U^{s-m}\seminorm[W^{s,p}(U)]{v}.
  \end{equation}
  The result follows from \eqref{interp.Wsp.1}, \eqref{interp.Wsp.2} and \eqref{interp.Wsp.3}.\end{proof}

\Wkpinterptrace*

\begin{proof}

  As expected $A\lesssim B$ is understood here up to a multiplicative
  constant that only depends on $N$, $\varrho$, $k$, $s$ and $p$.
  We first recall a classical continuous trace inequality:
  \begin{equation}\label{ineq.cont.trace}
    \forall w\in W^{1,p}(T)\,:\,h_T^{\frac1p}\norm[L^p(\partial T)]{w}\lesssim 
    \norm[L^p(T)]{w}+h_T\norm[L^p(T)]{\GRAD w}.
  \end{equation}
  For $p=2$ this inequality can be deduced from \cite[Lemma~1.49]{Di-Pietro.Ern:12}
  and many other references. The case of a general $p$ is less easy to find in the literature,
  but actually very simple to prove. Since $T$ is the union of disjoint simplices of inradius
  and diameter comparable to $h_T$, it is sufficient to prove the result when $T$ is one
  of these simplices $S$. For such a simplex, there exists an affine mapping $A:T\to T_0$,
  where $T_0=\{\vec{x}\in\Real^d\,:\,x_i> 0\,,\;\sum_{i=1}^d x_i<1\}$ is the
  reference simplex, such that the norms of the linear parts of $A$ and $A^{-1}$ are
  respectively of order $h_T^{-1}$ and $h_T$. Consider then $w_0\in W^{1,p}(T_0)$
  defined by $w_0(\vec{x})=w(A^{-1}\vec{x})$. On $T_0$ we have a trace inequality
  \begin{equation}\label{ineq.cont.trace.0}
    \norm[L^p(\partial T_0)]{w_0}\le C_{d,p} (\norm[L^p(T_0)]{w_0}+\norm[L^p(T_0)]{\GRAD w_0}).
  \end{equation}
  By noticing that $|\GRAD w_0(\vec{x})|\lesssim h_T |(\GRAD w)(A^{-1}\vec{x})|$
  and using changes of variables $\vec{x}\mapsto \vec{y}=A\vec{x}$, 
  \eqref{ineq.cont.trace.0} gives \eqref{ineq.cont.trace}.

  Estimate \eqref{eq:approx.lproj.Wsp.trace} is an immediate consequence of \eqref{ineq.cont.trace}
  and of \eqref{eq:approx.lproj.Wsp}. For $m\le s-1$, by applying \eqref{ineq.cont.trace}
  to $w=\partial^\alpha(v-\lproj[T]{k}v)\in W^{1,p}(T)$ for all $\alpha\in \mathbb{N}^N$
  of total length $m$ we find
  \[
  h_T^{\frac1p}\seminorm[{W^{m,p}(\Fh[T])}]{v-\lproj[T]{k}v}\lesssim 
  \seminorm[W^{m,p}(T)]{v-\lproj[T]{k}v}+h_T\seminorm[W^{m+1,p}(T)]{v-\lproj[T]{k}v}.
  \]
  We then use \eqref{eq:approx.lproj.Wsp} for $m$ and $m+1$ on the two terms
  in the right-hand side to conclude.\end{proof}

\paragraph{Acknowledgements.} This work was partially supported by ANR project HHOMM (ANR-15-CE40-0005).

\footnotesize
\bibliographystyle{plain}
\bibliography{hho-plap}

\begin{thebibliography}{10}

\bibitem{Andreianov.Boyer.ea:04}
B.~Andreianov, F.~Boyer, and F.~Hubert.
\newblock Finite volume schemes for the $p$-{Laplacian} on {Cartesian} meshes.
\newblock {\em ESAIM: Math. Model Numer. Anal. (M2AN)}, 38:931--954, 2004.

\bibitem{Andreianov.Boyer.ea:05}
B.~Andreianov, F.~Boyer, and F.~Hubert.
\newblock {Besov} regularity and new error estimates for finite volume
  approximations of the $p$-{Laplacian}.
\newblock {\em Numer. Math.}, 100:565--592, 2005.

\bibitem{Andreianov.Boyer.ea:06}
B.~Andreianov, F.~Boyer, and F.~Hubert.
\newblock On the finite-volume approximation of regular solutions of the
  $p$-{Laplacian}.
\newblock {\em IMA J. Numer. Anal.}, 26:472--502, 2006.

\bibitem{Andreianov.Boyer.ea:07}
B.~Andreianov, F.~Boyer, and F.~Hubert.
\newblock {Discrete Duality Finite Volume} schemes for {Leray--Lions}-type
  elliptic problems on general {2D} meshes.
\newblock {\em Num. Meth. PDEs}, 23:145--195, 2007.

\bibitem{Antonietti.Bigoni.ea:14}
P.~F. Antonietti, N.~Bigoni, and M.~Verani.
\newblock Mimetic finite difference approximation of quasilinear elliptic
  problems.
\newblock {\em Calcolo}, 52:45--67, 2014.

\bibitem{Antonietti.Giani.ea:13}
P.~F. Antonietti, S.~Giani, and P.~Houston.
\newblock $hp$-version composite discontinuous {G}alerkin methods for elliptic
  problems on complicated domains.
\newblock {\em SIAM J. Sci. Comput.}, 35(3):A1417--A1439, 2013.

\bibitem{Araya.Harder.ea:13}
R.~Araya, C.~Harder, D.~Paredes, and F.~Valentin.
\newblock Multiscale hybrid-mixed method.
\newblock {\em SIAM J. Numer. Anal.}, 51(6):3505--3531, 2013.

\bibitem{Arnold:82}
D.~N. Arnold.
\newblock An interior penalty finite element method with discontinuous
  elements.
\newblock {\em SIAM J. Numer. Anal.}, 19:742--760, 1982.

\bibitem{Bank.Yserentant:14}
R.~E. Bank and H.~Yserentant.
\newblock On the {$H^1$}-stability of the {$L_2$}-projection onto finite
  element spaces.
\newblock {\em Numer. Math.}, 126(2):361--381, 2014.

\bibitem{Barrett.Liu:94}
J.W. Barrett and W.~Liu.
\newblock Finite element approximation of degenerate quasi-linear elliptic and
  parabolic problems.
\newblock {\em Pitman Res. Notes Math. Ser.}, 303:1--16, 1994.

\bibitem{Bassi.Botti.ea:12}
F.~Bassi, L.~Botti, A.~Colombo, D.~A. Di~Pietro, and P.~Tesini.
\newblock On the flexibility of agglomeration based physical space
  discontinuous {Galerkin} discretizations.
\newblock {\em J. Comput. Phys.}, 231(1):45--65, 2012.

\bibitem{Beirao-da-Veiga.Brezzi.ea:13}
L.~Beir\~{a}o~da Veiga, F.~Brezzi, A.~Cangiani, G.~Manzini, L.~D. Marini, and
  A.~Russo.
\newblock Basic principles of virtual element methods.
\newblock {\em Math. Models Methods Appl. Sci. (M3AS)}, 199(23):199--214, 2013.

\bibitem{Beirao-da-Veiga.Brezzi.ea:13*1}
L.~Beir\~{a}o~da Veiga, F.~Brezzi, and L.~D. Marini.
\newblock Virtual elements for linear elasticity problems.
\newblock {\em SIAM J. Numer. Anal.}, 2(51):794--812, 2013.

\bibitem{Blatter.95}
H.~Blatter.
\newblock Velocity and stress fields in grounded glacier: a simple algorithm
  for including deviator stress gradients.
\newblock {\em J. Glaciol.}, 41:333--344, 1995.

\bibitem{Boccardo.et.al:92}
L.~Boccardo, Gallou\"et T., and F.~Murat.
\newblock Unicit\'e de la solution de certaines \'equations elliptiques non
  lin\'eaires.
\newblock {\em C.R. Acad. Sci. Paris}, 315:1159--1164, 1992.

\bibitem{Bramble.Pasciak.et.al:02}
J.~H. Bramble, J.~E. Pasciak, and O.~Steinbach.
\newblock On the stability of the {$L^2$} projection in {$H^1(\Omega)$}.
\newblock {\em Math. Comp.}, 71(237):147--156 (electronic), 2002.

\bibitem{Brenner:03}
S.~C. Brenner.
\newblock {P}oincar\'e-{F}riedrichs inequalities for piecewise {$H^1$}
  functions.
\newblock {\em SIAM J. Numer. Anal.}, 41(1):306--324, 2003.

\bibitem{Brenner.Scott:08}
S.~C. Brenner and L.~R. Scott.
\newblock {\em The mathematical theory of finite element methods}, volume~15 of
  {\em Texts in Applied Mathematics}.
\newblock Springer, New York, third edition, 2008.

\bibitem{Brezzi.Falk.ea:14}
F.~Brezzi, R.~S. Falk, and L.~D. Marini.
\newblock Basic principles of mixed virtual element methods.
\newblock {\em ESAIM: Math. Model. Numer. Anal.}, 48(4):1227--1240, 2014.

\bibitem{Brezzi.Lipnikov.et.al:05}
F.~Brezzi, K.~Lipnikov, and V.~Simoncini.
\newblock A family of mimetic finite difference methods on polygonal and
  polyhedral meshes.
\newblock {\em Math. Models Methods Appl. Sci.}, 15(10):1533--1551, 2005.

\bibitem{Buffa.Ortner:09}
A.~Buffa and C.~Ortner.
\newblock Compact embeddings of broken {S}obolev spaces and applications.
\newblock {\em IMA J. Numer. Anal.}, 4(29):827--855, 2009.

\bibitem{Burman.Ern:08}
E.~Burman and A.~Ern.
\newblock Discontinuous {Galerkin} approximation with discrete variational
  principle for the nonlinear {Laplacian}.
\newblock {\em C. R. Acad. Sci. Paris, Ser. I}, 346:1013--1016, 2008.

\bibitem{Carstensen:02}
C.~Carstensen.
\newblock Merging the {B}ramble--{P}asciak--{S}teinbach and the
  {C}rouzeix--{T}hom\'ee criterion for {$H^1$}-stability of the
  {$L^2$}-projection onto finite element spaces.
\newblock {\em Math. Comp.}, 71(237):157--163, 2002.

\bibitem{CasadoDiaz.et.al07}
J.~Casado-Diaz, F.~Murat, and A.~Porretta.
\newblock Uniqueness results for pseudomonotone problems with $p > 2$.
\newblock {\em C. R. Math. Acad. Sci. Paris}, 344(8):487--492, 2007.

\bibitem{Castillo.Cockburn.ea:00}
P.~Castillo, B.~Cockburn, I.~Perugia, and D.~{Sch{\"o}tzau}.
\newblock An a priori error analysis of the local discontinuous {G}alerkin
  method for elliptic problems.
\newblock {\em SIAM J. Numer. Anal.}, 38:1676--1706, 2000.

\bibitem{Cockburn.Di-Pietro.ea:15}
B.~Cockburn, D.~A. Di~Pietro, and A.~Ern.
\newblock Bridging the {Hybrid High-Order} and {Hybridizable Discontinuous
  Galerkin} methods.
\newblock {\em ESAIM: Math. Model Numer. Anal. (M2AN)}, 2015.
\newblock Published online.
  DOI~\href{http://dx.doi.org/10.1051/m2an/2015051}{10.1051/m2an/2015051}.

\bibitem{Cockburn.Gopalakrishnan.ea:09}
B.~Cockburn, J.~Gopalakrishnan, and R.~Lazarov.
\newblock Unified hybridization of discontinuous {G}alerkin, mixed, and
  continuous {G}alerkin methods for second order elliptic problems.
\newblock {\em SIAM J. Numer. Anal.}, 47(2):1319--1365, 2009.

\bibitem{Crouzeix.Thomee:87}
M.~Crouzeix and V.~Thom{\'e}e.
\newblock The stability in {$L_p$} and {$W^1_p$} of the {$L_2$}-projection onto
  finite element function spaces.
\newblock {\em Math. Comp.}, 48(178):521--532, 1987.

\bibitem{Deimling:95}
K.~Deimling.
\newblock {\em Nonlinear functional analysis}.
\newblock Springer-Verlag, Berlin, 1985.

\bibitem{DiPietro.Droniou.Ern:15}
D.~A. Di~Pietro, J.~Droniou, and A.~Ern.
\newblock A discontinuous-skeletal method for advection--diffusion--reaction on
  general meshes.
\newblock {\em SIAM J. Numer. Anal.}, 53(5):2135--2157, 2015.

\bibitem{Di-Pietro.Ern:10}
D.~A. Di~Pietro and A.~Ern.
\newblock Discrete functional analysis tools for discontinuous {G}alerkin
  methods with application to the incompressible {N}avier--{S}tokes equations.
\newblock {\em Math. Comp.}, 79:1303--1330, 2010.

\bibitem{Di-Pietro.Ern:12}
D.~A. Di~Pietro and A.~Ern.
\newblock {\em Mathematical aspects of discontinuous {G}alerkin methods},
  volume~69 of {\em Math\'ematiques \& Applications}.
\newblock Springer-Verlag, Berlin, 2012.

\bibitem{Di-Pietro.Ern.ea:14}
D.~A. Di~Pietro, A.~Ern, and S.~Lemaire.
\newblock An arbitrary-order and compact-stencil discretization of diffusion on
  general meshes based on local reconstruction operators.
\newblock {\em Comput. Meth. Appl. Math.}, 14(4):461--472, 2014.

\bibitem{Di-Pietro.Lemaire:14}
D.~A. Di~Pietro and S.~Lemaire.
\newblock An extension of the {Crouzeix--Raviart} space to general meshes with
  application to quasi-incompressible linear elasticity and {Stokes} flow.
\newblock {\em Math. Comp.}, 84(291):1--31, 2015.

\bibitem{Diaz.Thelin:94}
J.~I. Diaz and F.~de~Thelin.
\newblock On a nonlinear parabolic problem arising in some models related to
  turbulent flows.
\newblock {\em SIAM J. Math. Anal.}, 25(4):1085--1111, 1994.

\bibitem{Droniou:06}
J.~Droniou.
\newblock Finite volume schemes for fully non-linear elliptic equations in
  divergence form.
\newblock {\em ESAIM: Math. Model Numer. Anal. (M2AN)}, 40:1069--1100, 2006.

\bibitem{Droniou.Eymard:06}
J.~Droniou and R.~Eymard.
\newblock A mixed finite volume scheme for anisotropic diffusion problems on
  any grid.
\newblock {\em Numer. Math.}, 105:35--71, 2006.

\bibitem{Droniou.Eymard:09}
J.~Droniou and R.~Eymard.
\newblock Study of the mixed finite volume method for {S}tokes and
  {N}avier-{S}tokes equations.
\newblock {\em Numer. Methods Partial Differential Equations}, 25(1):137--171,
  2009.

\bibitem{koala}
J.~Droniou, R.~Eymard, T.~Gallou\"et, C.~Guichard, and R.~Herbin.
\newblock Gradient schemes for elliptic and parabolic problems.
\newblock 2015.
\newblock In preparation.

\bibitem{Droniou.Eymard.ea:10}
J.~Droniou, R.~Eymard, T.~Gallou\"{e}t, and R.~Herbin.
\newblock A unified approach to mimetic finite difference, hybrid finite volume
  and mixed finite volume methods.
\newblock {\em Math. Models Methods Appl. Sci. (M3AS)}, 20(2):1--31, 2010.

\bibitem{Doniou.Eymard.et.al:12}
J.~Droniou, R.~Eymard, T.~Gallou{\"e}t, and R.~Herbin.
\newblock Gradient schemes: a generic framework for the discretisation of
  linear, nonlinear and nonlocal elliptic and parabolic equations.
\newblock {\em Math. Models Methods Appl. Sci. (M3AS)}, 23(13):2395--2432,
  2012.

\bibitem{Dupont.Scott:80}
T.~Dupont and R.~Scott.
\newblock Polynomial approximation of functions in {S}obolev spaces.
\newblock {\em Math. Comp.}, 34(150):441--463, 1980.

\bibitem{Eymard.Gallouet.ea:10}
R.~Eymard, T.~Gallou{\"e}t, and R.~Herbin.
\newblock Discretization of heterogeneous and anisotropic diffusion problems on
  general nonconforming meshes. {SUSHI}: a scheme using stabilization and
  hybrid interfaces.
\newblock {\em IMA J. Numer. Anal.}, 30(4):1009--1043, 2010.

\bibitem{Girault.Riviere.ea:05}
V.~Girault, B.~Rivi{\`e}re, and M.~F. Wheeler.
\newblock A discontinuous {G}alerkin method with nonoverlapping domain
  decomposition for the {S}tokes and {N}avier-{S}tokes problems.
\newblock {\em Math. Comp.}, 74(249):53--84, 2005.

\bibitem{Glowinski:84}
R.~Glowinski.
\newblock {\em Numerical methods for nonlinear variational problems}.
\newblock Springer Series in Computational Physics. Springer-Verlag, New York,
  1984.

\bibitem{Glowinski.Rappaz:03}
R.~Glowinski and J.~Rappaz.
\newblock Approximation of a nonlinear elliptic problem arising in a
  non-{Newtonian} fluid flow model in glaciology.
\newblock {\em ESAIM: Math. Model Numer. Anal. (M2AN)}, 37(1):175--186, 2003.

\bibitem{Herbin.Hubert:08}
R.~Herbin and F.~Hubert.
\newblock Benchmark on discretization schemes for anisotropic diffusion
  problems on general grids.
\newblock In R.~Eymard and J.-M. H\'{e}rard, editors, {\em Finite Volumes for
  Complex Applications V}, pages 659--692. John Wiley \& Sons, 2008.

\bibitem{Karakashian.Jureidini:98}
O.~A. Karakashian and W.~N. Jureidini.
\newblock A nonconforming finite element method for the stationary
  {N}avier-{S}tokes equations.
\newblock {\em SIAM J. Numer. Anal.}, 35(1):93--120, 1998.

\bibitem{Lasis.Suli:03}
A.~Lasis and E.~S{\"u}li.
\newblock {P}oincar\'e-type inequalities for broken {S}obolev spaces.
\newblock Technical Report 03/10, Oxford University Computing Laboratory,
  Oxford, England, 2003.

\bibitem{Leray.Lions:65}
J.~Leray and J.-L. Lions.
\newblock Quelques r{\'e}sultats de {V}i\v sik sur les probl{\`e}mes
  elliptiques non lin{\'e}aires par les m{\'e}thodes de {M}inty-{B}rowder.
\newblock {\em Bull. Soc. Math. France}, 93:97--107, 1965.

\bibitem{Lipnikov.Manzini:14}
K.~Lipnikov and G.~Manzini.
\newblock A high-order mimetic method on unstructured polyhedral meshes for the
  diffusion equation.
\newblock {\em J. Comput. Phys.}, 272:360--385, 2014.

\bibitem{Liu.Yan:01}
W.~Liu and N.~Yan.
\newblock Quasi-norm a priori and a posteriori error estimates for the
  nonconforming approximation of $p$-{Laplacian}.
\newblock {\em Numer. Math.}, 89:341--378, 2001.

\bibitem{Minty:63}
G.~J. Minty.
\newblock On a ``monotonicity'' method for the solution of non-linear equations
  in {Banach} spaces.
\newblock {\em Proc. Nat. Acad. Sci. U.S.A.}, 50:1038--1041, 1963.

\bibitem{Wang.Ye:13}
J.~Wang and X.~Ye.
\newblock A weak {Galerkin} element method for second-order elliptic problems.
\newblock {\em J. Comput. Appl. Math.}, 241:103--115, 2013.

\bibitem{Wang.Ye:14}
J.~Wang and X.~Ye.
\newblock A weak {G}alerkin mixed finite element method for second order
  elliptic problems.
\newblock {\em Math. Comp.}, 83(289):2101--2126, 2014.

\end{thebibliography}

\end{document}